\documentclass [a4paper, 11pt] {article}
\usepackage{lipsum}
\usepackage{amsfonts}
\usepackage{bm}
\usepackage{amsmath}
\usepackage{graphicx,color}
\usepackage{epstopdf}
\usepackage[caption=false]{subfig}
\usepackage{float}
\usepackage{algorithmic}
\usepackage{amsopn}
\newtheorem{theorem}{Theorem}[section]
\newtheorem{corollary}[theorem]{Corollary}
\newtheorem{remark}[theorem]{Remark}
\newtheorem{lemma}[theorem]{Lemma}
\newtheorem{proposition}[theorem]{Proposition}
\newtheorem{example}[theorem]{Example}
\newenvironment{proof}[1][Proof]{\noindent\textbf{#1.} }{\ \rule{0.5em}{0.5em}}

\begin{document}
\begin{center}
{\bf \Large Convergence of a Semi-Discrete  Numerical Method for a Class of Nonlocal Nonlinear Wave Equations}
\\ ~ \\ \vspace*{20pt}
H. A. Erbay$^{1}$, S. Erbay$^{1}$,  A. Erkip$^{2}$
\vspace*{20pt}

$^{1}$Department of Natural and Mathematical Sciences, Faculty of Engineering, Ozyegin University,  Cekmekoy 34794, Istanbul, Turkey
\vspace*{20pt}

$^{2}$Faculty of Engineering and Natural Sciences, Sabanci University,  Tuzla 34956,  Istanbul,    Turkey

\end{center}

\footnote{E-mail:   husnuata.erbay@ozyegin.edu.tr, saadet.erbay@ozyegin.edu.tr, \\
albert@sabanciuniv.edu}

\begin{abstract}
In this article, we prove the convergence of a semi-discrete numerical method applied to a general class of nonlocal nonlinear wave equations where the nonlocality is introduced through the convolution operator in space. The most important characteristic of  the numerical method is that it is directly applied to the nonlocal  equation by introducing the discrete convolution operator. Starting from the continuous Cauchy problem  defined on the real line, we first construct the discrete Cauchy problem on a uniform grid of the real line. Thus the semi-discretization in space of the continuous problem gives rise to an infinite system of  ordinary differential equations in time. We show that the initial-value problem for this system is well-posed.  We prove that solutions of the discrete problem converge uniformly to those of the continuous one as the mesh size goes to zero and that  they are  second-order convergent in space.   We then consider a truncation of the infinite domain to a finite one. We prove that the solution of the truncated problem approximates the solution of the continuous problem when the truncated domain is sufficiently large.  Finally, we present some numerical experiments that confirm numerically both  the expected convergence rate of the semi-discrete scheme and the ability of the method to capture finite-time blow-up of solutions for various convolution kernels.
\end{abstract}
\vspace*{20pt}

\setcounter{equation}{0}
\section{Introduction}\label{sec1}

In this study we are interested in approximating solutions of the following class of nonlocal nonlinear wave equations
\begin{equation}
     u_{tt} =\left(\beta \ast  f(u)\right)_{xx},  \label{eq:cont}
\end{equation}%
where the symbol $\ast$ is used to denote the convolution operation in the spatial domain
 \begin{displaymath}
    (\beta \ast v)(x)= \int_{\mathbb{R}} \beta(x-y)v(y)\mbox{d}y
\end{displaymath}
and the kernel $\beta$ is an even function  with $\int_{\mathbb{R}}\beta(x)dx=1$.  For the initial-value problem of (\ref{eq:cont}), we present a second-order semi-discrete  scheme based on a uniform spatial discretization and prove its convergence.

Equation (\ref{eq:cont}) was proposed in \cite{Duruk2010} as a model for the propagation of strain waves in a one-dimensional, homogeneous, nonlinearly and nonlocally elastic infinite medium.  In the same paper some  local existence, global existence and blow-up results were established for  the initial-value problem of (\ref{eq:cont}).  The class (\ref{eq:cont}) covers a variety of equations from integro-differential equations to differential-difference equations that arise in lattice models \cite{Duruk2010}. Equation (\ref{eq:cont}) includes  some well-known nonlinear wave equations as a particular case. For instance, with  the exponential kernel  $\beta(x)={1\over 2}e^{-|x|}$  and $f(u)=u+g(u)$,  (\ref{eq:cont}) reduces to  the improved Boussinesq (IB) equation
\begin{equation}
    u_{tt}-u_{xx}-u_{xxtt}=(g(u))_{xx}  \label{eq:imbq}
\end{equation}
that has been widely investigated in the literature. If $\beta$ is the  Green's function for the differential operator $1-L(D^{2}_{x})$ where $D_{x}$ represents the partial derivative with respect to $x$, (\ref{eq:cont}) reduces to the higher-order IB-type equation
\begin{equation}
    u_{tt} -u_{xx}-L(D^{2}_{x})u_{tt}=g(u)_{xx}. \label{eq:nondif}
\end{equation}
We note that, in a more general case, $L$ corresponding to $\beta$ will be a pseudo-differential operator.

For the IB equation and its higher-order versions,  several numerical schemes have been  developed  in the literature. They include   finite difference schemes \cite{Bratsos2007,Wang2014,Zhang2012} and spectral methods \cite{Borluk2015,Oruc2017} among many others. Clearly, the schemes that replace derivatives by finite-differences will not work for (\ref{eq:nondif}) in the case where $L$ is a general pseudo-differential operator.   In order to solve  (\ref{eq:cont}) for a general arbitrary kernel function there has been one recent attempt \cite{Borluk2017} in which  a pseudospectral Fourier method in space has been developed. In \cite{Borluk2017} the authors have proved the convergence of the semi-discrete pseudospectral Fourier method and they have tested their method on two problems:  propagation of a single solitary wave and finite time blow-up of solutions. For both of the problems they have observed a good  agreement between the numerical and analytical results. From a practical point of view, we remark that the method proposed in \cite{Borluk2017} can be used if  the Fourier transform of the kernel function is known. This is the main drawback for the numerical method since, in the most cases, the kernel function is given  in physical space rather than Fourier space. Because of the mathematical difficulties associated with the convolution integrals involving a general arbitrary kernel function, the direct numerical approximation of the nonlocal equation (\ref{eq:cont}) is a difficult task and is highly demanding from applications'  standpoint. Additionally, to the best of our knowledge, no efforts have been made yet to solve numerically the  initial-value problem of (\ref{eq:cont}) with a general arbitrary kernel function using spatial discretization methods.  These motivate us to develop a convergent semi-discrete scheme that can be directly applied to (\ref{eq:cont}). Several  numerical studies where some other nonlocal evolution problems are solved by direct computation  are available in the literature (see, among the others, \cite{Emmrich2007b, Emmrich2007a, Du2010, Du2013, Guan2015} for a linear peridynamic model of elasticity  and \cite{Bates2009, Armstrong2010, Rossi2011} for  nonlinear nonlocal parabolic models).

In the present study we first construct a discrete initial-value problem of (\ref{eq:cont}) on a uniform grid of the real line. This is achieved by  transferring the spatial derivatives to $\beta$ and then discretizing the convolution integral. We note that our semi-discretization does not involve any spatial discrete derivative of $u$.   The semi-discrete problem is in fact an infinite dimensional system of ordinary differential equations in time, where the mesh size appears as a parameter. We consider the semi-discrete system corresponding to the continuous initial-value problem for which the solution exists on some time interval.  For sufficiently small mesh sizes, we prove that the discretized solutions exist on the same time interval and converge to the continuous solution at mesh points as the mesh size goes to zero. Moreover the error in the approximation is quadratic with respect to the mesh size. Even though both the continuous problem and the discrete problem are of infinite extent, the computations based on the semi-discrete scheme must be done in a finite domain. To resolve this issue, we consider a truncated problem on a finite interval and obtain the corresponding finite dimensional system of ordinary differential equations.  As expected, the error involved in this truncation is related to the dimension of the final system. We  prove that the error resulting from the truncation turns out to depend on the decay behavior of the original solution of the continuous problem. We then illustrate these issues in two cases; propagation of the solitary wave for the IB equation and finite-time blow-up of solutions for (\ref{eq:cont}) with various kernels. In the numerical examples we observe both the quadratic rate of convergence with respect to the mesh size and the effect of the size of the computational domain on the error.

Several points are worth noting. First, even for the IB equation or its higher order versions, our method may be more efficient than finite-difference methods since it  involves  discretization of integrals rather than  discretization of derivatives. Second, if $L$ in (\ref{eq:nondif}) is a pseudo-differential operator, the finite-difference methods may not be suitable.  Third,  one of the issues  resolved by  our method  is the computational complexity resulting from the fact that the discrete system is also nonlocal as in the continuous problem. Fourth, the method developed in  \cite{Borluk2017} needs the Fourier transform of the kernel, and numerical computation for the Fourier transform of the kernel leads to further errors. Fifth, as our approach does not utilize discretizations of spatial derivatives, a further time discretization will not involve any stability issues regarding spatial mesh size. We therefore prefer to stop at the semi-discrete level.

We also want to note that our approach can be adopted  to unidirectional nonlocal wave equations of the form
\begin{displaymath}
    u_{t} =\left( \beta \ast  f(u)\right)_{x}
\end{displaymath}
which generalize the Benjamin-Bona-Mahony (BBM) equation \cite{Benjamin1972}.  A numerical sch\-eme based on the discretization of an integral representation of the solution was used in \cite{Bona1981} to solve the BBM equation. The starting point of our numerical method is similar to that in  \cite{Bona1981}.  Also, our remark about the stability issue regarding time discretization was already observed in \cite{Bona1981}.

The rest of the paper is organized as follows. In Section \ref{sec:sec2}, we briefly describe the main features of discretization and present some preliminary lemmas.  In Section \ref{sec:sec3} we introduce the infinite-dimensional semi-discrete problem and establish the local well-posedness of the related initial-value problem. In Section \ref{sec:sec4} we prove the convergence of solutions of the semi-discrete  problem to the solutions of the continuous one  with second-order accuracy in space. By truncating the semi-discrete problem, a finite-dimensional system of ordinary differential equations  is introduced in Section \ref{sec:sec5} and it is shown that the solution of the truncated system approximates the solution of the continuous problem. In Section \ref{sec:sec6} some numerical experiments  are conducted  to verify our theoretical findings.

Throughout the paper, we use the standard notation for function spaces. The notation $\| u\|_{L^p}$ denotes the $L^p$ ($1\leq p \leq \infty$) norm of $u$ on $\mathbb{R}$. The symbol $\langle u, v\rangle$ represents the inner product of $u$ and $v$ in $L^2$. The notation $W^{k,p}(\mathbb{R})=\{ u\in L^p(\mathbb{R}): D^ju \in L^p(\mathbb{R}),~~ j \leq k \} $ denotes the $L^{p}$-based Sobolev space with the norm $\| u\|_{W^{k,p}}=\sum_{j\leq k} \|D^j u\|_{L^p},~1\leq p \leq \infty$. The symbol $H^{s}$ is the usual Sobolev space of index $s$ on $\mathbb{R}$. We will drop the symbol $\mathbb{R}$ in $\int_{\mathbb{R}}$. The symbol $C$ will stand for a generic positive constant.

\setcounter{equation}{0}
\section{Discretization and Preliminary Lemmas}\label{sec:sec2}

In this section we will derive error estimates for discretizations of integrals and derivatives on an infinite grid. On a finite interval these estimates are standard but usually depend on the length of the interval. As we are dealing with the whole space $\mathbb{R}$, we will need the estimates  in terms of $L^{p}$ norms. Below we state several lemmas whose proofs follow more or less standard lines. For completeness, the proofs are given in Appendix  \ref{sec:appendixA}.

We  use bold letter for two-sided infinite sequences $\mathbf{u}=(u_{i})_{i=-\infty}^{i=\infty}=(u_{i})$ where $u_{i}$ is the $i$th component of $\mathbf{u}$. For a fixed $h>0$ and $1\leq p<\infty $, the space $l_{h}^{p}$ is defined as
\begin{displaymath}
   l_{h}^{p}=l_{h}^{p}\left(\mathbb{Z}\right)=\left\{ (u_{i}): u_{i}\in \mathbb{R}, ~~
        \Vert \mathbf{u}\Vert_{l_{h}^{p}}^{p}=\sum_{i=-\infty }^{\infty}h|u_{i}|^{p}\right\}.
\end{displaymath}
Clearly, $l_{h}^{2}$ is a Hilbert space with inner product
\begin{displaymath}
    \langle \mathbf{u},\mathbf{v}\rangle _{l_{h}^{2}}=\sum_{i}hu_{i}v_{i}
\end{displaymath}
(here and henceforth, unless otherwise stated,  the summation index $i$ runs over the set $\mathbb{Z}$ of integers). Similarly, $l^{\infty}$ denotes the Banach space $l^{\infty}(\mathbb{Z})$ with the norm $\displaystyle \Vert \mathbf{u}\Vert_{l^{\infty}}=\sup_{i \in\mathbb{Z}} \left \vert u_{i} \right \vert$. For two sequences  $\mathbf{u}$ and $\mathbf{v}$   we define the discrete convolution as
\begin{equation}
    (\mathbf{u}\ast \mathbf{v})_{i}=\sum_{j}hu_{i-j}v_{j}. \label{eq:disc-con}
\end{equation}%
As in the continuous case, when $\mathbf{u}\in l_{h}^{1}, \mathbf{v}\in l_{h}^{p}$, $1\leq p<\infty $,  we have Young's inequality $\Vert \mathbf{u}\ast \mathbf{v}\Vert_{l_{h}^{p}}\leq \Vert \mathbf{u}\Vert _{l_{h}^{1}}\Vert \mathbf{v}\Vert _{l_{h}^{p}}.$ For $\mathbf{u}\in l_{h}^{1},\mathbf{v}\in l^{\infty }$, we also have $\Vert \mathbf{u}\ast \mathbf{v}\Vert _{l^{\infty} }\leq \Vert \mathbf{u}\Vert_{l_{h}^{1}}\Vert \mathbf{v}\Vert_{l^{\infty}}$.

We now consider the grid points $x_{i}=ih$, $i\in \mathbb{Z}$ with the mesh size $h$  over $\mathbb{R}$. We  define the restriction and extension operators between functions and sequences below \cite{Amorim2013}. For a function $u$ on $\mathbb{R}$ the restriction operator is $\mathbf{R}u=(u(x_{i}))$. When there is no danger of confusion we will write $\mathbf{u}$ instead of $ \mathbf{R}u$ and use the abbreviations $\mathbf{u}^{\prime }=\mathbf{R}u^{\prime }$,  $\mathbf{u}^{\prime\prime }=\mathbf{R}u^{\prime\prime }$ and so on. For a sequence $\mathbf{u}$ we define the piecewise constant extension operator $P_{0}(\mathbf{u})(x)=u_{i}$ for $x_{i}\leq x<x_{i+1}$, $i\in \mathbb{Z}$ and the piecewise linear extension operator
\begin{displaymath}
    P_{1}(\mathbf{u})(x)=u_{i}+\frac{u_{i+1}-u_{i}}{h}(x-x_{i})\text{\ \ for \ }x_{i}\leq x<x_{i+1}, ~~~i\in \mathbb{Z}.
\end{displaymath}
The following lemma gives discretization errors for integrals on $\mathbb{R}$.
\begin{lemma}\label{lem:lem2.1}
    Suppose $1\leq p < \infty $, $u\in W^{1,p}(\mathbb{R})$ and $\mathbf{u}=\mathbf{R}u$. Then $\mathbf{u}\in l_{h}^{p}$ and
    \begin{equation}
        \Vert u-P_{0}(\mathbf{u})\Vert _{L^{p}}\leq h\Vert u^{\prime }\Vert _{L^{p}}. \label{eq:est1}
    \end{equation}%
    Moreover, if $u\in W^{2,p}(\mathbb{R})$ then
    \begin{equation}
        \Vert u-P_{1}(\mathbf{u})\Vert _{L^{p}}\leq h^{2}\Vert u^{\prime \prime }\Vert_{L^p}.  \label{eq:est2}
    \end{equation}
\end{lemma}
\begin{remark}\label{rem:rem2.2}
    We want to emphasize that Lemma \ref{lem:lem2.1}  fails without the smoothness assumption. The following example can be given. Let
    \begin{displaymath}
        u(x)=\sum_{n=1}^{\infty}\chi_{(n-\frac{1}{n^{2}},n+\frac{1}{n^{2}})}(x),
    \end{displaymath}
    where \ $\chi _{A}$ denotes the characteristic function of the set $A$. Then $\Vert \mathbf{u}\Vert_{l_{h}^{1}}=\infty $  for any rational  $h>0$ while $u\in L^{1}(\mathbb{R})$.
\end{remark}
\begin{remark}\label{rem:rem2.3}
Clearly the integrals $\int P_{0}(\mathbf{u})(x)dx$ and $\int P_{1}(\mathbf{u})(x)dx$ correspond respectively to the rectangular and trapezoidal approximations of the integral $\int u(x)dx$ on $\mathbb{R}$. Yet we note that
\begin{displaymath}
    \int P_{0}(\mathbf{u})(x)dx=\int P_{1}(\mathbf{u})(x)dx=\sum_{i} hu(x_{i}).
\end{displaymath}
Thus (\ref{eq:est1}) and (\ref{eq:est2}) in Lemma \ref{lem:lem2.1}  can be interpreted as the discrete estimates of the integral $\int u(x)dx$ for the rectangular and trapezoidal rules, respectively.
\end{remark}
\begin{remark}\label{rem:rem2.4}
    For $p=1$, the  estimate (\ref{eq:est2})  in Lemma \ref{lem:lem2.1} extends to the case $u^{\prime \prime }=\mu $ where $\mu$ is a finite measure on $\mathbb{R}$. In that case  the estimate becomes
    \begin{displaymath}
        \Vert u-P_{1}(\mathbf{u})\Vert _{L^{1}}\leq h^{2}|\mu |(\mathbb{R}).
     \end{displaymath}
\end{remark}

Combining Remark \ref{rem:rem2.3} with Remark \ref{rem:rem2.4} we get the following bounds on the discretization of the integral on $\mathbb{R}$.
\begin{corollary}\label{cor:cor2.5}
    Let $u\in W^{1,1}(\mathbb{R})$. Then
    \begin{displaymath}
        \left\vert \int u(x)dx-\sum_{i}h u(x_{i})\right \vert \leq h\Vert u^{\prime }\Vert _{L^{1}}.
    \end{displaymath}%
    Moreover, if $u^{\prime \prime }=\mu $ is a finite measure on\ $\mathbb{R}$, then
    \begin{displaymath}
        \left\vert \int u(x)dx-\sum_{i}h u(x_{i})\right \vert \leq h^{2}|\mu|(\mathbb{R}).
    \end{displaymath}
\end{corollary}
We note that Corollary \ref{cor:cor2.5} covers the regular case $u\in W^{2,1}(\mathbb{R})$ where $d\mu =u^{\prime\prime} dx$.

On $l_{h}^{p}$ we define the difference operators
\begin{equation}
    (D^{+}\mathbf{u})_{i}=\frac{1}{h}(u_{i+1}-u_{i}),~~~~(D^{-}\mathbf{u})_{i}=\frac{1}{h}(u_{i}-u_{i-1})
\end{equation}%
for which we have $\langle D^{\pm }\mathbf{u},\mathbf{v}\rangle _{l_{h}^{2}}=-\langle \mathbf{u},D^{\mp }\mathbf{v}\rangle _{l_{h}^{2}}$.  Moreover, a direct computation shows that the discrete convolution (\ref{eq:disc-con}) commutes with the difference operators $D^{\pm }$,
\begin{displaymath}
    D^{\pm }(\mathbf{u}\ast \mathbf{v})=(D^{\pm }\mathbf{u})\ast \mathbf{v}=\mathbf{u}\ast (D^{\pm }\mathbf{v}).
\end{displaymath}

The following two lemmas give  $l^{\infty}$ estimates on the discrete derivatives.
\begin{lemma}\label{lem:lem2.6}
    Let $u\in W^{2,\infty}(\mathbb{R})$, $\mathbf{u}=\mathbf{R}u$ and $\mathbf{u}^{\prime }=\mathbf{R}u^{\prime }.$ Then
    \begin{displaymath}
        \Vert D^{\pm }\mathbf{u}-\mathbf{u}^{\prime }\Vert_{l^{\infty}}\leq \frac{h}{2}\Vert u^{\prime \prime }\Vert_{L^{\infty}}.
    \end{displaymath}
\end{lemma}
We next consider the second-order difference operator defined as
\begin{displaymath}
    \left(D^{+}D^{-}\mathbf{u}\right)_{i}={1\over h^{2}}\left(u_{i+1}-2u_{i}+u_{i-1}\right).
\end{displaymath}
Clearly $D^{+}D^{-}=D^{-}D^{+}$.
\begin{lemma}\label{lem:lem2.7}
    Let $u\in W^{4,\infty}(\mathbb{R})$, $\mathbf{u}=\mathbf{R}u$ and $\mathbf{u}^{\prime \prime }=\mathbf{R}u^{\prime \prime }$. Then
    \begin{displaymath}
        \Vert D^{+}D^{-}\mathbf{u}-\mathbf{u}^{\prime \prime }\Vert _{l^{\infty}}\leq \frac{h^{2}}{12}\Vert u^{(4)}\Vert _{L^{\infty}}.
    \end{displaymath}
\end{lemma}

From the above representations of the difference operators, we get similar estimates in the $l_{h}^{p}$ norms.
\begin{lemma}\label{lem:lem2.8}
    Let $u\in W^{2, p}(\mathbb{R})$, $\mathbf{u}=\mathbf{R}u$ and $\mathbf{u}^{\prime }=\mathbf{R}u^{\prime }.$ Then
    \begin{displaymath}
        \Vert D^{\pm }\mathbf{u}-\mathbf{u}^{\prime }\Vert _{l_{h}^{p}}\leq C h\Vert u^{\prime \prime }\Vert _{L^{P}}.
    \end{displaymath}
\end{lemma}
\begin{lemma}\label{lem:lem2.9}
    Let $u\in W^{4, p}(\mathbb{R})$, $\mathbf{u}=\mathbf{R}u$ and $\mathbf{u}^{\prime \prime }=\mathbf{R}u^{\prime \prime }$.
    Then
    \begin{displaymath}
        \Vert D^{+}D^{-}\mathbf{u}-\mathbf{u}^{\prime \prime}\Vert_{l_{h}^{p}}\leq C h^{2}\Vert u^{(4)}\Vert_{L^{p}}.
    \end{displaymath}
\end{lemma}

\setcounter{equation}{0}
\section{The Continuous and Discrete Cauchy Problems}\label{sec:sec3}

We will consider the Cauchy problem
\begin{align}
   & u_{tt} =(\beta \ast f(u))_{xx}\text{ \ \ \ }x\in \mathbb{R}\text{, \ \ }t>0 \label{eq:cont1} \\
   & u(x,0)=\varphi(x), ~~~~~ u_{t}(x,0)=\psi(x)  \text{\ \ \ }x\in \mathbb{R}.  \label{eq:initial}
\end{align}%
We assume that $f$ is sufficiently smooth with $f(0)=0$ and that the kernel $\beta $ satisfies the following:

\begin{enumerate}
    \item $\beta $ $\in W^{1,1}(\mathbb{R})$

    \item $\beta ^{\prime \prime }=\mu $ is a finite Borel measure on $\mathbb{R}$.
\end{enumerate}
We note that Condition 2 above also includes the more regular case $\beta $ $\in W^{2,1}(\mathbb{R})$; i.e. $\beta ^{\prime \prime }\in L^{1}(\mathbb{R})$ with $d\mu=\beta ^{\prime \prime }dx$. Moreover, these conditions imply that the Fourier transform of the kernel is of the form $\widehat{\beta }(\xi )={\cal O} \left((1+\xi^{2})^{-1}\right)$. This in turn suffices to show the local well-posedness of the Cauchy problem of (\ref{eq:cont1})-(\ref{eq:initial}).  On the other hand, for  the typical kernel $\beta(x)={1\over 2}e^{-|x|}$ we have $\beta^{\prime\prime}=\beta-\delta $ with the Dirac measure $\delta$. This explains why we impose Condition 2 (see \cite{Duruk2010} for details).
\begin{theorem}\label{theo:theo3.1} (Theorem 3.4 and Lemma 3.9 of \cite{Duruk2010})
    Let $f\in C^{\lfloor s\rfloor +1}(\mathbb{R})$ with $f(0)=0,$ $s>\frac{1}{2}$. For given $\varphi ,\psi \in H^{s}(\mathbb{R})$, there is some $T>0$ so that the initial-value problem (\ref{eq:cont1})-(\ref{eq:initial}) is locally well-posed with solution $u\in C^{2}\left([0,T],H^{s}(\mathbb{R})\right)$. Moreover, there is a global solution if and only if for any $T < \infty $ we have
    \begin{displaymath}
        \limsup_{t\rightarrow T^{-}}\left\Vert u\left( t\right) \right\Vert_{L^{\infty} }< \infty ~.
    \end{displaymath}
\end{theorem}
\begin{remark}\label{rem:blow}
    The proof of Theorem \ref{theo:theo3.1} relies on two main ingredients;  the  $H^{s}$-valued ODE character of (\ref{eq:cont1}) and local Lipschitz estimates for the nonlinear term. We want to emphasize the second assertion of the theorem; it implies that if blow-up occurs it should be observed in $\left\Vert u\left( t\right) \right\Vert _{L^{\infty} }$. Hence one can not have higher order singularities if the amplitude stays finite. This  is due to the $L^{\infty}$ control of the nonlinear term, given in the following lemma  \cite{Cons2002, Meyer1997}.
\end{remark}
\begin{lemma}\label{lem:yeni}
    Let $s\geq 0,$ $f\in C^{[s]+1}(\mathbb{R})$ with $f(0)=0$. Then for any $u\in H^{s}\cap L^{\infty }$, we have $f(u)\in H^{s}\cap L^{\infty }$. Moreover there is some constant $C(M)$ depending on $M$ such that for all $u\in H^{s}\cap L^{\infty }$ with $\left\Vert u\right\Vert _{L^{\infty} }\leq M$
    \begin{displaymath}
        {\left\Vert f(u)\right\Vert }_{H^{s}}\leq C(M){\left\Vert u\right\Vert }_{H^{s}}~.
    \end{displaymath}
\end{lemma}

In order to define the related semi-discrete problem, we fix $h>0$ and discretize  (\ref{eq:cont1}) as follows:
\begin{equation}
    \frac{d^{2}\mathbf{v}}{dt^{2}} =D^{+}D^{-}\big(\bm{\beta}_{h}\ast f(\mathbf{v})\big)  \label{eq:disc}
\end{equation}%
where  $f(\mathbf{v})=(f(v_{i}))$ and $\bm{\beta}_{h}=\mathbf{R}\beta$ is the restriction (hence discretization) of the kernel $\beta $.

Noting that $D^{+}D^{-}(\bm{\beta}_{h}\ast f(\mathbf{v}))=$ $(D^{+}D^{-}\bm{\beta}_{h})\ast f(\mathbf{v})$, we first estimate $D^{+}D^{-}\bm{\beta}_{h}$.
\begin{lemma}\label{lem:lem3.2}
    $D^{+}D^{-}\bm{\beta}_{h}\in l_{h}^{1}$ and $\Vert D^{+}D^{-}\bm{\beta}_{h}\Vert _{l_{h}^{1}}\leq 2|\mu| (\mathbb{R}).$
\end{lemma}
\begin{proof}
    First we assume that $\beta ^{\prime \prime }\in L^{1}$. Then
    \begin{align}
        h^{2}(D^{+}D^{-}\bm{\beta}_{h})_{i}
        &=\beta (x_{i+1})-2\beta (x_{i})+\beta (x_{i-1}) \nonumber \\
        &=\int_{x_{i}}^{x_{i+1}}\beta^{\prime }(s)ds-\int_{x_{i-1}}^{x_{i}}\beta^{\prime }(s)ds \nonumber \\
        &=\int_{x_{i-1}}^{x_{i}}\Big(\beta ^{\prime }(s+h)-\beta ^{\prime }(s)\Big)ds \nonumber \\
        &=\int_{x_{i-1}}^{x_{i}}\int_{s}^{s+h}\beta ^{\prime \prime }(r)drds. \nonumber
    \end{align}%
    As $x_{i-1}\leq s\leq s+h\leq x_{i+1}$
    \begin{displaymath}%
        \left\vert h(D^{+}D^{-}\bm{\beta }_{h})_{i}\right\vert
            \leq {1\over h}\int_{x_{i-1}}^{x_{i}}\int_{x_{i-1}}^{x_{i+1}}\left\vert \beta^{\prime\prime}(r)\right\vert drds
            =\int_{x_{i}-1}^{x_{i}+1}\left\vert \beta^{\prime \prime }(r)\right\vert dr,
    \end{displaymath}
    \begin{displaymath}
        \Vert D^{+}D^{-}\bm{\beta }_{h}\Vert_{l_{h}^{1}}
            =\sum_{i}h \left\vert (D^{+}D^{-}\bm{\beta }_{h})_{i}\right\vert
            \leq \sum_{i}\int_{x_{i}-1}^{x_{i+1}}\left\vert \beta^{\prime \prime }(r)\right\vert dr
            \leq 2\Vert\beta ^{\prime \prime }\Vert _{L^{1}}.
    \end{displaymath}%
    When $\beta^{\prime \prime }=\mu $ is a finite measure, this estimate becomes
    $\Vert D^{+}D^{-}\bm{\beta }_{h}\Vert _{l_{h}^{1}}\leq 2|\mu| (\mathbb{R})$.
\end{proof}

\begin{theorem}\label{theo:theo3.3}
    Let $f$ \ be a locally Lipschitz function with $f(0)=0$. Then the initial-value problem for (\ref{eq:disc}) is locally well-posed for initial data $\mathbf{v}(0)$, $\mathbf{v}^{\prime}(0)$ in $l^{\infty} $. Moreover there exists some maximal time $T_{h}>0$ so that the problem has unique solution $\mathbf{v}\in C^{2}([0,T_{h}),l^{\infty})$. The maximal time $T_{h}$, if finite, is determined by the blow-up condition
    \begin{equation}
        \limsup_{t\rightarrow T_{h}^{-}}\Vert \mathbf{v}(t)\Vert_{l^{\infty} }=\infty .  \label{eq:blow2}
    \end{equation}
\end{theorem}
\begin{proof}
We will consider (\ref{eq:disc}) as an $\l^{\infty}$-valued ordinary differential equation and apply Picard's Theorem on Banach spaces. To that end we first note that if $\Vert \mathbf{v}\Vert_{l^{\infty}}\leq M$ and $\Vert \mathbf{w}\Vert_{l^{\infty}}\leq M$, then
     \begin{equation}
      \Vert f(\mathbf{v})-f(\mathbf{w})\Vert_{l^{\infty}}\leq L_{M}\Vert \mathbf{v}-\mathbf{w}\Vert _{l^{\infty}},
    \end{equation}
    where   $L_{M}$ is the Lipschitz constant of $f$ on $[-M, M]$. By Young's inequality and Lemma \ref{lem:lem3.2} we have
    \begin{align}
        \Vert D^{+}D^{-}\big( \bm{\beta }_{h}\ast f(\mathbf{v})\big)-D^{+}D^{-}\big(\bm{\beta }_{h}\ast f(\mathbf{w})\big)\Vert_{l^{\infty}}
        & \leq \Vert D^{+}D^{-}\bm{\beta }_{h}\Vert_{l_{h}^{1}}\Vert f(\mathbf{v})-f(\mathbf{w})\Vert _{l^{\infty}} \nonumber \\
        & \leq 2\left \vert \mu \right \vert (\mathbb{R})L_{M}\Vert \mathbf{v}-\mathbf{w}\Vert _{l^{\infty}}.
    \end{align}%
    Hence the map
    \begin{displaymath}
        \mathbf{v}\longrightarrow D^{+}D^{-}\big(\bm{\beta}_{h}\ast f(\mathbf{v})\big)
    \end{displaymath}%
    is locally Lipschitz on $l^{\infty}$. By Picard's Theorem on Banach spaces, this implies the local well-posedness of the initial-value problem for (\ref{eq:disc}).     Standard theory of ordinary differential equations gives the blow-up condition as
    \begin{equation}
        \limsup_{t\rightarrow T_{h}^{-}}\left( \Vert \mathbf{v}(t)\Vert_{l^{\infty} }+ \Vert \mathbf{v}^{\prime}(t)\Vert_{l^{\infty} }\right)=\infty .  \label{eq:blow2a}
    \end{equation}
    To complete the proof we have to show that $\Vert \mathbf{v}^{\prime}(t)\Vert_{l^{\infty}}$ will not blow-up unless $\Vert \mathbf{v}(t)\Vert_{l^{\infty} }$ does so.    For that, suppose $\limsup_{t\rightarrow T_{h}^{-}}\Vert \mathbf{v}(t)\Vert _{l^{\infty} }=M<\infty $. Then, integrating (\ref{eq:disc}) we have
    \begin{displaymath}
         \mathbf{v}^{\prime }(t) =\mathbf{v}^{\prime}(0)+\int_{0}^{t}D^{+}D^{-}\Big(\bm{\beta}_{h}\ast f(\mathbf{v}(s))\Big) ds
    \end{displaymath}%
    so that by Lemma \ref{lem:lem3.2} for $t<T_{h}$,
    \begin{align}
        \Vert \mathbf{v}^{\prime }(t)\Vert_{l^{\infty}}
        &\leq \Vert\mathbf{v}^{\prime}(0)\Vert _{l^{\infty}}
            +2\left \vert \mu \right \vert (\mathbb{R})L_{M}\int_{0}^{t}\Vert \mathbf{v}(s)\Vert _{l^{\infty}}ds   \nonumber \\
        &\leq \Vert\mathbf{v}^{\prime}(0)\Vert _{l^{\infty}} +2\left \vert \mu \right \vert (\mathbb{R})L_{M} MT_{h} < \infty.
    \end{align}%
\end{proof}
\begin{remark}
    Theorem \ref{theo:theo3.3} is on the local well-posedness of the Cauchy problem for (\ref{eq:disc}) with a fixed value of $h$.  In the next section, when we consider the family of discretized problems  corresponding to the continuous problem (\ref{eq:cont1})-(\ref{eq:initial}), Theorem \ref{theo:theo4.1} will provide a uniform bound on $T_{h}$.
\end{remark}

\setcounter{equation}{0}
\section{ Discretization Error}\label{sec:sec4}

In this section our aim is to prove that  the discrete solution will approximate the continuous one when the discrete initial data is taken as the discretization of the continuous data. To be precise, we start with the solution $u\in C^{2}\left([0,T],H^{s}(\mathbb{R})\right)$ of (\ref{eq:cont1})-(\ref{eq:initial}) with sufficiently large $s $. We denote the discretizations of the continuous initial data $\varphi, \psi$ by  $\bm{\varphi}_{h}=\mathbf{R}\varphi$, $\bm{\psi}_{h}=\mathbf{R}\psi $. Let $\mathbf{u}_{h}\in C^{2}\left([0,T_{h}),l^{\infty}\right)$ be the solution of  (\ref{eq:disc}) with the initial data $\bm{\varphi}_{h},\bm{\psi}_{h}$. Our aim is to prove the following theorem:
\begin{theorem}\label{theo:theo4.1}
    Let $s>\frac{9}{2}$, $ f\in C^{\lfloor s\rfloor +1}(\mathbb{R})$ with $f(0)=0$, $~\varphi,~\psi \in H^{s}(\mathbb{R})$, and let $u\in C^{2}\left([0,T], H^{s}(\mathbb{R})\right)$ be the solution of the initial-value problem (\ref{eq:cont1})-(\ref{eq:initial}). Let
       $\mathbf{u}_{h}\in C^{2}\left([0,T_{h}),l^{\infty}\right)$ be the solution of  (\ref{eq:disc}) with initial data $\bm{\varphi}_{h},\bm{\psi}_{h}$.   Let $\mathbf{u}(t)=\mathbf{R}u(t) =(u(x_{i},t))$. Then there is some $h_{0}$ so that for $h\leq h_{0}$, the maximal existence time $T_{h}$ of $\mathbf{u}_{h}$ is at least $T$ and
    \begin{equation}
        \Vert \mathbf{u}(t)-\mathbf{u}_{h}(t)\Vert_{l^{\infty}}+\Vert \mathbf{u}_{t}(t)-\mathbf{u}_{h}^{\prime }(t)\Vert  \label{eq:fourone} _{l^{\infty}}=\mathcal{O}(h^{2})
    \end{equation}
    for all $t\in \lbrack 0,T \rbrack$.
\end{theorem}
\begin{proof}
    We first let $\displaystyle M=\max_{0\leq t\leq T}\Vert u(t)\Vert _{L^{\infty}}$. Since $\Vert \bm{\varphi }_{h}\Vert _{l^{\infty }}\leq \Vert \varphi \Vert_{L^{\infty }}\leq M$, by continuity there is some maximal time $t_{h}\leq T$ such that $\Vert \mathbf{u}_{h}(t)\Vert _{l^{\infty}}\leq 2M$ for all $t\in \lbrack 0,t_{h}]$. Moreover, by the maximality condition either $t_{h}=T$ or $\Vert \mathbf{u}_{h}(t_{h})\Vert _{l^{\infty}}=2M$. \ At the point $\ x=x_{i}$,  (\ref{eq:cont1}) becomes
    \begin{displaymath}
        u_{tt}(x_{i}, t) =\big(\beta \ast f(u)\big)_{xx}(x_{i}, t) .\text{ }
    \end{displaymath}%
     Recalling that $\mathbf{u}(t)=\mathbf{R}u(t)$, this becomes $\mathbf{u}^{\prime\prime}(t)=\mathbf{R}\big(\beta \ast f(u)\big)_{xx}(t)$.  A residual term $\mathbf{F}_{h}$ arises from the discretization of the right-hand side of (\ref{eq:cont1}):
     \begin{equation}
       {{d^{2}\mathbf{u}}\over {dt^{2}}}= D^{+}D^{-}\big(\bm{\beta }_{h}\ast f(\mathbf{u})\big)+\mathbf{F}_{h}, \label{eq:u-F}
    \end{equation}%
     where
     \begin{displaymath}
        \mathbf{F}_{h}\mathbf{=R}\big(\beta \ast f(u)\big)_{xx}-D^{+}D^{-}\big(\bm{\beta }_{h}\ast f(\mathbf{u})\big).
     \end{displaymath}
     The $i$th entry  satisfies
    \begin{align}
        (F_{h})_{i} &=\big(\beta \ast f(u)\big)_{xx}(x_{i})-D^{+}D^{-}\big(\bm{\beta}_{h}\ast f(\mathbf{u})\big)_{i} \nonumber  \\
        &=\Big(\big(\beta \ast f(u)_{xx}\big)(x_{i})-\big( \bm{\beta}_{h}\ast \mathbf{R}f(u)_{xx}\big)_{i}\Big)
         +\Big(\big(\bm{\beta}_{h}\ast \mathbf{R}f(u)_{xx}\big)_{i}-\big(\bm{\beta}_{h}\ast D^{+}D^{-}f(\mathbf{u})\big)_{i}\Big) \nonumber \\
        &=(F_{h}^{1})_{i}+(F_{h}^{2})_{i}, \nonumber
    \end{align}%
    where the variable $t$ is suppressed for brevity.     We start with the  term $(F_{h}^{1})_{i}$. Replacing $f(u)$ by $g$ for convenience, we have
    \begin{displaymath}
        (F_{h}^{1})_{i}=(\beta \ast g^{\prime\prime})(x_{i})-(\bm{\beta}_{h}\ast \mathbf{g}^{\prime\prime})_{i}=\int \beta (x_{i}-y)g^{\prime\prime}(y)dy-\sum_{j}h\beta (x_{i}-x_{j})g^{\prime\prime}(x_{j}).
    \end{displaymath}%
    We first assume that $\beta \in W^{2,1}(\mathbb{R})$. By Corollary \ref{cor:cor2.5} we have
    \begin{displaymath}
        \left \vert (F_{h}^{1})_{i}\right \vert \leq h^{2}\Vert r^{\prime \prime }\Vert _{L^{1}}
    \end{displaymath}%
    where $r(y)=\beta (x_{i}-y)g^{\prime\prime}(y)$. Then
    \begin{displaymath}
        r^{\prime \prime }(y)=\beta ^{\prime \prime }(x_{i}-y)g^{\prime\prime}(y)-2\beta ^{\prime}(x_{i}-y)g^{\prime\prime\prime }(y)+\beta (x_{i}-y)g^{(4)}(y).
    \end{displaymath}%
    By Lemma \ref{lem:yeni} we have
    \begin{displaymath}
    \Vert g\Vert_{H^{s}}=\Vert f(u)\Vert _{H^{s}}\leq C(M)\Vert u\Vert _{H^{s}}.
    \end{displaymath}
    For $s>\frac{9}{2}$ we observe that $g^{\prime\prime}$, $g^{\prime\prime\prime}$ and $g^{(4)}$ are bounded, so that $r^{\prime \prime }\in L^{1}(\mathbb{R})$. In case $\beta^{\prime \prime }=\mu $ is a finite measure, then $r^{\prime \prime }=\widetilde{\mu }$ will be a measure with
    \begin{displaymath}
        |\widetilde{\mu }|(\mathbb{R})\leq C\Big(|\mu |(\mathbb{R})+2\Vert \beta \Vert_{W^{1,1}}\Big)\Vert u\Vert _{H^{s}},
    \end{displaymath}%
    so that
    \begin{displaymath}
        |(F_{h}^{1})_{i}|\leq h^{2}|\widetilde{\mu }|(\mathbb{R}).
    \end{displaymath}%
    For the second term $(F_{h}^{2})_{i}$, again with $g=f(u)$ and $s>\frac{9}{2}$ we have
    \begin{align}
        \left \vert (F_{h}^{2})_{i}\right \vert
        &=\left \vert \left( \bm{\beta}_{h}\ast (\mathbf{R}g^{\prime \prime}-D^{+}D^{-}\mathbf{g})\right)_{i}\right\vert
            \leq \Vert \bm{\beta}_{h}\Vert_{l_{h}^{1}}\Vert \mathbf{g}^{\prime \prime }-D^{+}D^{-}\mathbf{g}\Vert _{l^{\infty }} \nonumber \\
        &\leq h^{2}\Vert g^{(4)}\Vert _{L^{\infty }}
            \leq h^{2}\Vert f(u)\Vert_{W^{4,\infty }}
            \leq Ch^{2}\Vert f(u)\Vert _{H^{s}}  \nonumber \\
         &   \leq C(M)h^{2}\Vert u\Vert _{H^{s}},  \nonumber
    \end{align}%
    where Lemmas \ref{lem:lem2.7} and \ref{lem:yeni} are used.    Combining the estimates for $|(F_{h}^{1})_{i}|$ and $|(F_{h}^{2})_{i}|$, we obtain
    \begin{displaymath}
        \Vert \mathbf{F}_{h}(t)\Vert _{l^{\infty }}\leq C  h^{2}\Vert u(t)\Vert_{H^{s}},
    \end{displaymath}
    where $C=C(\beta, M)$ depends on the bounds on $\beta$ and $M=\sup_{0\leq t \leq T}\Vert u(t)\Vert_{L^{\infty}}$.     We now let $\mathbf{e}(t)=\mathbf{u}(t)-\mathbf{u}_{h}(t)$  be the error term. Then, from (\ref{eq:disc}) and (\ref{eq:u-F}) we have
    \begin{align}
        {{d^{2}\mathbf{e}(t)}\over {dt^{2}}} &=D^{+}D^{-}\bm{\beta}_{h}\ast \big(f(\mathbf{u})-f(\mathbf{u}_{h})\big)+\mathbf{F}_{h} \nonumber \\
        \mathbf{e}(0) &=\mathbf{0},\text{\ \ \ \ \ }\mathbf{e}^{\prime }(0)=\mathbf{0}.  \nonumber
    \end{align}%
    This implies
    \begin{align}
        \mathbf{e}(t)&=
            \int_{0}^{t}(t-\tau )\Big(D^{+}D^{-}\bm{\beta}_{h}\ast \big(f(\mathbf{u})-f(\mathbf{u}_{h})\big)+\mathbf{F}_{h}\Big)d\tau\\
         \mathbf{e}^{\prime }(t)&=
            \int_{0}^{t}\Big(D^{+}D^{-}\bm{\beta}_{h}\ast \big(f(\mathbf{u})-f(\mathbf{u}_{h})\big)+\mathbf{F}_{h}\Big) d\tau .
    \end{align}%
    But $\left \Vert f(\mathbf{u})-f(\mathbf{u}_{h})\right \Vert_{l^{\infty}} \leq L_{2M}\Vert \mathbf{u}-\mathbf{u}_{h}\Vert _{l^{\infty }}$, so that for $t\leq t_{h}\leq T,$
    \begin{align}
        \Vert \mathbf{e}(t)\Vert _{l^{\infty }}+\Vert \mathbf{e}^{\prime }(t)\Vert_{l^{\infty }}
           & \leq (1+T)\Big(C Th^{2}+L_{2M}\Vert D^{+}D^{-}\bm{\beta}_{h}\Vert _{l_{h}^{1}}\int_{0}^{t}\Vert \mathbf{e}(\tau )\Vert_{l^{\infty }}d\tau\Big), \nonumber  \\
           & \leq C(\beta, M) (1+T) \Big(\sup_{0\leq t \leq T}\Vert u(t)\Vert_{H^{s}} Th^{2}
           +\int_{0}^{t}\Vert \mathbf{e}(\tau )\Vert_{l^{\infty }}d\tau \Big), \nonumber \\
           & \leq C(\beta, M) (1+T) \Big(\sup_{0\leq t \leq T}\Vert u(t)\Vert_{H^{s}} Th^{2}
            +\int_{0}^{t}\big(\Vert \mathbf{e}(\tau )\Vert_{l^{\infty }}+\Vert \mathbf{e}^{\prime}(\tau )\Vert_{l^{\infty }}\big)d\tau \Big),
    \end{align}%
    where the second inequality follows from the bound for $D^{+}D^{-}\bm{\beta}_{h}$ in Lemma \ref{lem:lem3.2}. Then, by Gronwall's inequality,
    \begin{displaymath}
        \Vert \mathbf{e}(t)\Vert _{l^{\infty }}+\Vert \mathbf{e}^{\prime }(t)\Vert _{l^{\infty }}
            \leq Ch^{2} (1+T)T \sup_{0\leq t \leq T}\Vert u(t)\Vert_{H^{s}}e^{C(1+T)t}
    \end{displaymath}
    with $C=C(\beta, M)$.   This, in particular, implies that  $\Vert \mathbf{e}(t_{h})\Vert _{l^{\infty }}<M$ for sufficiently small $h$. Then we have $\Vert \mathbf{u}_{h}(t_{h})\Vert _{l^{\infty }}<2M$ showing that $t_{h}=T_{h}=T$. From the above estimate we get (\ref{eq:fourone}).
\end{proof}
\begin{remark}
    The proof above gives the estimate
    \begin{displaymath}
        \left\Vert \mathbf{u(}t\mathbf{)-u}_{h}\mathbf{(}t\mathbf{)}\right\Vert_{l_{\infty }}+\left\Vert \mathbf{u}_{t}\mathbf{(}t\mathbf{)-u}_{h}^{\prime }\mathbf{(}t\mathbf{)}\right\Vert _{l_{\infty }}
        \leq Ch^{2} (1+T)T \sup_{0\leq t \leq T}\Vert u(t)\Vert_{H^{s}}e^{C(1+T)t},
    \end{displaymath}
    where $C=C(\beta, M)$. Obviously, $C$ depends on the estimates of the kernel $\beta$, the solution $u$, and the nonlinear term $f(u)$. Noting that  the nonlinear bound depends on  $u$, $T$ that are determined   from the initial data (\ref{eq:initial}) by Theorem \ref{theo:theo3.1}, we can say that
    \begin{displaymath}
        \left\Vert \mathbf{u(}t\mathbf{)-u}_{h}\mathbf{(}t\mathbf{)}\right\Vert_{l_{\infty }}+\left\Vert \mathbf{u}_{t}\mathbf{(}t\mathbf{)-u}_{h}^{\prime }\mathbf{(}t\mathbf{)}\right\Vert _{l_{\infty }}\leq C(\beta, u,f, T)h^{2}.
    \end{displaymath}
\end{remark}
\begin{remark}\label{rem:rem4.2}
    The usual approach in obtaining error estimates for the discretized problem is via the use of a suitable energy identity or energy inequality.  Nevertheless,  defining a reasonable energy may be difficult for nonlocal problems. In our case, for the continuous problem  (\ref{eq:cont1})-(\ref{eq:initial}), one can define a conserved energy by inverting the convolution operator $\beta \ast ( . )$. Yet, this requires further assumptions on $\beta$ and is not easy to carry over the same approach to the discrete case; because it necessitates the inversion of the matrix $B=\left(b_{ij}\right)$ with $b_{ij}=\beta(x_{i}-x_{j})$,  $i, j\in \mathbb{Z}$. This is why we follow the above direct approach.
\end{remark}

\setcounter{equation}{0}
\section{The Truncated Problem}\label{sec:sec5}

In this section we investigate the truncated finite dimensional system
\begin{equation}
    \frac{d^{2}v_{i}^{N}}{dt^{2}}
     =\sum_{j=-N}^{N}hD^{+}D^{-}\beta (x_{i}-x_{j})f(v_{j}^{N}),\text{\ \ \ \ \ \ }-N\leq i\leq N,  \label{eq:trunca}
\end{equation}%
which is obtained by considering the first $2N+1$ rows and columns of  (\ref{eq:disc}). We express (\ref{eq:trunca}) as
\begin{displaymath}
    \frac{d^{2}\mathbf{v}^{N}}{dt^{2}}=B^{N}f(\mathbf{v}^{N}),
\end{displaymath}
where $B^{N}$ is the $(2N+1)\times (2N+1)$ matrix with the entries $b^{N}_{ij}=hD^{+}D^{-}\beta (x_{i}-x_{j})$. We note that due to our formulation no boundary terms appear. By Lemma \ref{lem:lem3.2} we have
\begin{displaymath}
    \Vert B^{N}\mathbf{w}\Vert_{l^{\infty}}\leq  2|\mu| (\mathbb{R})\Vert \mathbf{w}\Vert_{l^{\infty}},
\end{displaymath}
where we use the norm $\displaystyle \Vert \mathbf{w}\Vert_{l^{\infty}}=\max_{-N \leq i \leq N} \left \vert w_{i} \right \vert$ for vectors in $\mathbb{R}^{2N+1}$. It follows from the smoothness of $f$ that the initial-value problem defined for  the above system of ordinary differential equations has a solution on $[0, T^{N})$. Moreover the argument in the proof of Theorem \ref{theo:theo3.3} shows that the blow-up condition is
\begin{equation}
        \limsup_{t\rightarrow (T^{N})^{-}} \Vert \mathbf{v}^{N}(t)\Vert_{l^{\infty} }=\infty .  \label{eq:blow3}
    \end{equation}
In this section we will show that, for sufficiently large $N$, the solution $\mathbf{v}^{N}=(v_{i}^{N})$ of (\ref{eq:trunca}) approximates  the solution $\mathbf{v}$ of the semi-discrete problem defined for (\ref{eq:disc}) and hence it approximates the solution $u$ of the continuous problem (\ref{eq:cont1})-(\ref{eq:initial}). We define a truncation operator ${\cal T}^{N}: l^{\infty} \rightarrow \mathbb{R}^{2N+1}$ as the projection  ${\cal T}^{N}\mathbf{v}=(v_{-N},v_{-N+1}, \ldots ,v_{0}, \ldots, v_{N-1},v_{N})$.
\begin{theorem}\label{theo:theo5.1}
    Let $\mathbf{v}\in C^{2}\left([0,T],l^{\infty}\right)$  be the solution of (\ref{eq:disc}) with initial values $\mathbf{v}(0)$, $\mathbf{v}^{\prime}(0)$ and let
    \begin{displaymath}
        \delta=\sup \big\{ \left\vert v_{i}(t)\right\vert : t\in \left[ 0,T\right] , \left\vert i\right\vert >N\big\} ~~\mbox{and}~~
         \epsilon (\delta )=\max_{\left\vert z\right\vert \leq \delta } \left\vert f(z)\right\vert .
    \end{displaymath}
    Then for sufficiently small $\epsilon(\delta)$, the solution $\mathbf{v}^{N}$ \ of (\ref{eq:trunca}) with initial  values $\mathbf{v}^{N}(0)={\cal T}^{N}\mathbf{v}(0)$, $(\mathbf{v}^{N})^{\prime}(0)={\cal T}^{N}\mathbf{v}^{\prime}(0)$ exists for  times $t\in \lbrack 0,T]$ and
    \begin{displaymath}
        \left\vert v_{i}^{N}(t)-v_{i}(t)\right\vert+\left\vert (v_{i}^{N})^{\prime}(t)-(v_{i})^{\prime}(t)\right\vert
            \leq C\epsilon (\delta ),\text{ \ }t\in \lbrack 0,T\rbrack,
    \end{displaymath}%
    for all $-N\leq i \leq N$.
\end{theorem}
\begin{proof}
    We follow the approach in the proof of Theorem \ref{theo:theo4.1}. Taking the components with $-N \leq i \leq N$ of (\ref{eq:disc}) we have
    \begin{align}
    \frac{d^{2}v_{i}}{dt^{2}}
     &=\sum_{j=-\infty}^{\infty}hD^{+}D^{-}\beta (x_{i}-x_{j})f(v_{j})  \nonumber \\
     &=\sum_{j=-N}^{N}hD^{+}D^{-}\beta (x_{i}-x_{j})f(v_{j})+F_{i}^{N} \nonumber
    \end{align}%
    with the residual term
    \begin{displaymath}
    F_{i}^{N}=\sum_{|j|> N}hD^{+}D^{-}\beta (x_{i}-x_{j})f(v_{j}).
    \end{displaymath}%
    Then  ${\cal T}^{N}\mathbf{v}$ satisfies the system
    \begin{displaymath}
        \frac{d^{2}{\cal T}^{N}\mathbf{v}}{dt^{2}}=B^{N}f({\cal T}^{N}\mathbf{v})+\mathbf{F}^{N}
    \end{displaymath}
    with the residual term $\mathbf{F}^{N}=(F_{i}^{N})$. Estimating the residual term we get
     \begin{displaymath}
        \left \vert F_{i}^{N}\right \vert \leq 2|\mu|(\mathbb{R}) \sup_{|j|> N} \left \vert f(v_{j})\right \vert \leq C \epsilon(\delta).
    \end{displaymath}%
     As in the proof of Theorem \ref{theo:theo4.1} we set  $\displaystyle M=\max_{0\leq t\leq T}\Vert \mathbf{v}(t)\Vert _{l^{\infty} }$. Since $\Vert\mathbf{v}(0)\Vert _{l^{\infty }}\leq M$, by continuity of the solution $\mathbf{v}^{N}$ of the truncated problem there is some maximal time $t_{N}\leq T$ such that we have   $\Vert \mathbf{v}^{N}(t)\Vert_{l^{\infty}}\leq 2M$ for all $t\in \lbrack 0,t_{N}]$. By the maximality condition either $t_{N}=T$ or $\Vert \mathbf{v}^{N}(t_{N})\Vert_{l^{\infty}}=2M$. We define the error term $\widetilde{\mathbf{e}}={\cal T}^{N}\mathbf{v}-\mathbf{v}^{N}$. Then
    \begin{align}
        &   \frac{d^{2} \widetilde{\mathbf{e}}(t)}{dt^{2}}
            =B^{N}\big(f({\cal T}^{N}\mathbf{v})-f(\mathbf{v}^{N})\big)+\mathbf{F}^{N}, \nonumber \\
        & \widetilde{\mathbf{e}}(0) =\mathbf{0},\text{\ \ \ \ }\widetilde{\mathbf{e}}^{\prime }(0)=\mathbf{0}, \nonumber
    \end{align}%
    so
    \begin{displaymath}
        \widetilde{\mathbf{e}}(t)=\int_{0}^{t}(t-\tau )\Big(B^{N}\big(f({\cal T}^{N}\mathbf{v})-f(\mathbf{v}^{N})\big)+\mathbf{F}^{N}\Big)d\tau .
    \end{displaymath}%
    Then
    \begin{displaymath}
        \Vert \widetilde{\mathbf{e}}(t)\Vert_{l^{\infty}}\leq 2|\mu| (\mathbb{R})\int_{0}^{t}(t-\tau )\left\Vert \big(f({\cal T}^{N}\mathbf{v})-f(\mathbf{v}^{N})\big)(\tau) \right\Vert_{l^{\infty}}d\tau + C t^{2}\epsilon(\delta) .
    \end{displaymath}%
    But
    \begin{displaymath}
        \left\Vert f({\cal T}^{N}\mathbf{v})-f(\mathbf{v}^{N})\right\Vert_{l^{\infty}}
            \leq L_{2M} \left\Vert {\cal T}^{N}\mathbf{v}-\mathbf{v}^{N}\right\Vert_{l^{\infty}},
    \end{displaymath}%
    where $L_{2M}$ is the Lipschitz constant for $f$ on $\left[ -2M,2M\right] $. Putting together, we have
    \begin{displaymath}
        \Vert \widetilde{\mathbf{e}}(t)\Vert _{l^{\infty }}\leq CT^{2}\epsilon
            +CT\int_{0}^{t}\Vert \widetilde{\mathbf{e}}(\tau )\Vert _{l^{\infty }}d\tau,
    \end{displaymath}%
    and by Gronwall's inequality,
    \begin{displaymath}
        \Vert \widetilde{\mathbf{e}}(t)\Vert _{l^{\infty }}\leq CT^{2}\epsilon e^{CT^{2}}.
    \end{displaymath}%
    This, in particular, implies that there is some $\epsilon_{0}$ such that for all $\epsilon(\delta)\leq \epsilon_{0} $ we have $\Vert \widetilde{\mathbf{e}}(t_{N})\Vert _{l^{\infty }}<M$. Then we have     $\Vert \mathbf{v}^{N} (t_{N}) \Vert_{l^{\infty }} < 2M$ showing that $t_{N}=T$. The required estimate for $\widetilde{\mathbf{e}}^{\prime }(t)$ follows from the identity
    \begin{displaymath}
         \widetilde{\mathbf{e}}^{\prime }(t)
            =\int_{0}^{t}\Big(B^{N}\big(f({\cal T}^{N}\mathbf{v})-f(\mathbf{v}^{N})\big)+\mathbf{F}^{N}\Big) d\tau
    \end{displaymath}
    and this completes the proof.
\end{proof}
\begin{theorem}\label{theo:theo5.2}
    Let $s>\frac{9}{2}$, $f\in C^{\lfloor s\rfloor +1}(\mathbb{R})$ with $f(0)=0$, $\varphi, \psi \in H^{s}(\mathbb{R})$, and let $u\in C^{2}\left([0,T],H^{s}(\mathbb{R})\right)$ \ be the solution  of the initial-value problem (\ref{eq:cont1})-(\ref{eq:initial}).  Then for sufficiently small $h$ and  $\epsilon >0$, there is an $N$ so that  the solution $\mathbf{u}_{h}^{N}$ \ of (\ref{eq:trunca}) with initial  values $\mathbf{u}_{h}^{N}(0)={\cal T}^{N}\bm{\varphi}_{h}$, $(\mathbf{u}_{h}^{N})^{\prime}(0)={\cal T}^{N}\bm{\psi}_{h}$ exists for  times $t\in \lbrack 0,T]$ and
    \begin{equation}
        \Big\vert u(ih,t) -\left(\mathbf{u}_{h}^{N}\right)_{i}(t)\Big\vert
            +\left \vert u_{t}(ih,t)-\left(\mathbf{u}_{h}^{N}\right)^{\prime}_{i}(t)\right \vert = {\cal O}\left(h^{2}+\epsilon\right),
            \text{ \ }  t\in \lbrack 0,T] \label{eq:fivetwo}
    \end{equation}%
    for all $-N\leq i\leq N$.
\end{theorem}

The proof follows from putting together the results of Theorems \ref{theo:theo4.1} and \ref{theo:theo5.1}. Let $u\in C^{2}\left([0,T],H^{s}(\mathbb{R)}\right)$ be the solution of (\ref{eq:cont1})-(\ref{eq:initial}). Then, by Theorem \ref{theo:theo4.1}, the solution $\mathbf{u}_{h}\in C^{2}\left([0,T],l^{\infty }\right)$ of the semi-discrete problem (\ref{eq:disc}) with initial data $\mathbf{u}_{h}(0)=\bm{\varphi}_{h}$, $\left( \mathbf{u}_{h}\right)^{\prime }(0)=\bm{\psi}_{h}$ satisfies
\begin{displaymath}
    \left\Vert \mathbf{u}(t)-\mathbf{u}_{h}(t)\right\Vert_{l^{\infty}}
        +\left\Vert \mathbf{u}_{t}(t)-(\mathbf{u}_{h})^{\prime }(t)\right\Vert_{l^{\infty }}={\cal O}\left(h^{2}\right)
\end{displaymath}
for all $t\in \left[0, T\right] $. This means
\begin{equation}
    \left \vert u(ih,t)-(\mathbf{u}_{h})_{i}(t)\right \vert
        +\left \vert u_{t}(ih,t)-(\mathbf{u}_{h})_{i}^{\prime}(t)\right \vert\leq C h^{2},  \text{ \ }i\in \mathbb{Z}   \label{eq:estimate}
\end{equation}%
 for all $t\in \left[ 0,T\right] $ and for sufficiently small $h$.

The next step is to approximate $\mathbf{u}_{h}$ by the solution $\mathbf{u}_{h}^{N}$ of the truncated problem  defined by (\ref{eq:trunca}) and the initial values
\begin{displaymath}
    \mathbf{u}_{h}^{N}(0) =\mathcal{T}^{N}\ \mathbf{u}_{h}(0) =\mathcal{T}^{N}\bm{\varphi}_{h},\text{ \ \ \ }
    \left( \mathbf{u}_{h}^{N}\right)^{\prime }(0) =\mathcal{T}^{N}(\mathbf{u}_{h})^{\prime }(0) =\mathcal{T}^{N}\bm{\psi}_{h}.
\end{displaymath}
To apply Theorem \ref{theo:theo5.1} for a given $\epsilon >0$, we should choose the appropriate value of $N$. This value will be determined from the condition $\displaystyle \epsilon =\epsilon (\delta )=\max_{\left\vert z\right\vert \leq \delta }\left\vert f(z)\right\vert $ where
\begin{displaymath}
    \delta =\sup \big\{ \left\vert \left(\mathbf{u}_{h}\big)_{i}(t)\right\vert :
        t\in \left[0, T\right] ,\text{ }\left\vert i\right\vert >N\right\} .
\end{displaymath}
To that end we will use the following observation.
\begin{proposition} \label{prop:prop5.3}
    Suppose $w$ is continuous on $\mathbb{R}\times \left[ 0,T\right] $. If $~\lim_{\left\vert x\right\vert \rightarrow \infty }w(x,t)=0$ for all  $t\in \left[ 0,T\right]$, then $~\lim_{\left\vert x\right\vert \rightarrow \infty }w(x,t)=0$ uniformly on $\left[ 0,T\right] .$
\end{proposition}
\begin{proof}
    Assuming the contrary, there is some  $\nu >0,$ and a sequence $(x_{n},t_{n})$ such that $\left\vert x_{n}\right \vert \rightarrow \infty $ and $\left \vert w(x_{n},t_{n})\right \vert \geq \nu$. Since $t_{n}\in [0, T]$, there is a subsequence $(t_{n_{k}})$  that converges to some $t_{0}\in [0, T]$. But then  $\left \vert w(x_{n_{k}},t_{0})\right \vert \geq \frac{\nu }{2}>0 $ for sufficiently large $ n_{k}$ contradicting $\lim_{\left\vert x\right\vert \rightarrow \infty }w(x,t_{0})=0.$
\end{proof}
We are now ready to complete the proof of Theroem \ref{theo:theo5.2}. Let $\epsilon >0$ be given. Since $f(0)=0$, by continuity there is some $\delta >0$ so that $\epsilon =\epsilon (\delta )=\max_{\left\vert z\right\vert \leq \delta}\left\vert f(z)\right\vert $. Moreover we can assume that  $\epsilon(\delta )$ is sufficiently small so that the conclusion of Theorem \ref{theo:theo5.1} holds (the proof of Theorem \ref{theo:theo5.1} shows that the bound $\epsilon_{0}$ for $\epsilon(\delta )$ does not depend on $N$). Next, since $u\in C^{2}\left([0,T],H^{s}(\mathbb{R)}\right)$ with $
s>\frac{9}{2}$, $u$ is continuous on $\mathbb{R}\times \left[ 0,T\right] $ and satisfies $\lim_{\left\vert x\right\vert \rightarrow \infty }u(x,t)=0$ for all $t\in \left[0 ,T\right]$. By Proposition \ref{prop:prop5.3},  $\lim_{\left\vert x\right\vert \rightarrow \infty }u(x,t)=0$ uniformly on $\left[0, T\right]$. Hence there is some $S>0$ so that $\left\vert u(x,t)\right\vert \leq \frac{\delta}{2}$ for all $\left\vert x\right\vert \geq S$ and $t\in \left[ 0,T\right]$. Let $N=\frac{1}{h}(\lfloor S\rfloor )$. Then for $\left\vert i\right\vert >N$ we have $\left\vert u\left( ih,t\right) \right\vert \leq \frac{\delta }{2}$. By the estimate (\ref{eq:estimate}), for sufficiently small $h$
\begin{displaymath}
    \left\vert (\mathbf{u}_{h})_{i}(t)\right\vert
        \leq \left\vert u\left(ih,t\right) \right\vert +\left\vert u\left( ih,t\right) -(\mathbf{u}_{h})_{i}(t)\right\vert
        \leq \frac{\delta }{2}+Ch^{2}\leq \delta
\end{displaymath}
for all $\left\vert i\right\vert >N.$ Now Theorem \ref{theo:theo5.1} applies to yield
\begin{displaymath}
    \left \vert (\mathbf{u}_{h})_{i}(t)-(\mathbf{u}_{h}^{N})_{i}(t)\right \vert
    +\left \vert (\mathbf{u}_{h})_{i}^{\prime }(t)-(\mathbf{u}_{h}^{N})_{i}^{\prime }(t)\right\vert \leq C\epsilon .
\end{displaymath}
Combining this result with (\ref{eq:estimate}) gives (\ref{eq:fivetwo}). This completes the proof of Theorem \ref{theo:theo5.2}.
\begin{remark}\label{rem:rem5.4}
    The proof of Theorem \ref{theo:theo5.2} shows that for a given $\epsilon>0$, the truncation parameter $N$ is determined by the relation
    \begin{displaymath}
        \epsilon=\max\big\{ |f(u(x,t))|: |x|\geq Nh, ~~t\in [0, T]\big\}.
    \end{displaymath}
    The existence of such $N$ is guaranteed by Proposition \ref{prop:prop5.3}. On the other hand,  if one has further information on the decay behavior of the solution $u(x,t)$, this  will in turn give a more explicit relation between $N$ and $\epsilon$. In Appendix \ref{sec:appendixB}, following the idea in \cite{Bona1981} we give a general decay estimate which applies for certain kernels $\beta$.
\end{remark}
\begin{remark}\label{rem:rem5.5}
    Clearly, instead of  the truncation  system (\ref{eq:trunca}) in which $-N \leq i \leq N$, we could also consider the truncation on any interval determined by $N_{0}\leq i \leq N_{1}$.   Likewise, if one a priori knows that the solution $u$ is concentrated on the interval $[a,b]$; namely that
    \begin{displaymath}
        \sup \big\{ \left\vert u(x,t)\right\vert : t\in \left[ 0,T\right] , x\notin [a, b]\big\}
    \end{displaymath}
    stays small, then Theorem \ref{theo:theo5.2} gives the same conclusion for the truncated system with $N_{0}\leq i \leq N_{1}$ where the integers $N_{0}$ and $N_{1}$ are given by $N_{0}=\lfloor {a\over h} \rfloor-1$ and  $N_{1}=\lfloor {b\over h} \rfloor+1$.  The typical example for such a behavior is when the solution $u$ is a traveling wave that is localized in space, then it would suffice to consider an asymmetric interval determined by  $N_{0}\leq i \leq N_{1}$.
\end{remark}
\begin{remark}\label{rem:rem5.6}
     Finally we want to consider the stability issue of our numerical scheme. Our approach does not involve discretizations of spatial derivatives of $u$, instead the spatial derivatives act upon the kernel $\beta$. This appears as the coefficient matrix $B^{N}$ in (\ref{eq:trunca}) and Lemma \ref{lem:lem3.2} shows that $B^{N}$ is uniformly bounded with respect to  the mesh size $h$. In other words, the factor $h^{-2}$ that appears in the standard discretization of the second derivatives is not effective in (\ref{eq:trunca}); it is already embedded in the uniform estimate of $B^{N}$. So if we further apply time discretization to (\ref{eq:trunca}) there will be no stability limitation regarding spatial mesh size. Roughly speaking, a fully discrete version of our numerical scheme will be straightforward. This is the main reason why we  prefer to use an ODE solver rather than a fully discrete scheme in the numerical experiments performed in the next section.
\end{remark}

\setcounter{equation}{0}
\section{Numerical Experiments}\label{sec:sec6}

In this section we focus on two numerical examples for the semi-discrete scheme developed in the previous sections.  The first example  is about the propagation of a single solitary wave and the second example is about finite time blow-up of solutions. The numerical examples support both the expected properties of the semi-discrete scheme and the theoretical findings for (\ref{eq:cont}). For instance, we numerically observe that the semi-discrete scheme  (\ref{eq:trunca}) is  second-order convergent in space and that (\ref{eq:cont}) has finite time blow-up solutions.

To integrate the semi-discrete system (\ref{eq:trunca}) in time we use a standard fourth-order Runge-Kutta scheme. In addition to this, instead of developing a fully discrete scheme  for (\ref{eq:trunca}),  we use, in all the experiments to be described, the Matlab solver \verb"ode45" that performs a direct numerical integration of a set of ordinary differential equations using the fourth-order Runge-Kutta method.  To ensure that the dominant error in the solution is not temporal we fix the relative and absolute tolerances for the solver \verb"ode45" to be $RelTol=10^{-10}$ and $AbsTol=10^{-10}$, respectively.

\subsection{Propagation of a single solitary wave}

We  illustrate the accuracy of our semi-discrete scheme by considering solitary wave solutions of the IB equation, a member of the class (\ref{eq:cont}), corresponding to the exponential kernel.  The IB equation  with quadratic nonlinearity has the single solitary wave solution
\begin{equation}
    u(x,t)=A~\text{sech}^{2}\big(B(x-ct-x_{0}) \big), \label{eq:solitary}
\end{equation}
with $A=3(c^{2}-1)/2$, $B=\sqrt{A}/(\sqrt{6}c)$ and $c^{2}>1$.  Equation (\ref{eq:solitary}) represents the solitary wave centered at $x_{0}$ initially and propagating in the positive direction of the $x$-axis with the constant wave speed $c$, the width $B^{-1}$ and the  amplitude $A$. We note that  the amplitude  and the width of the solitary wave are determined by  the wave speed. In (\ref{eq:solitary}), $u$  is a smooth function which decays exponentially  and approaches zero as $|x| \rightarrow \infty$, so the errors resulting from a suitable truncation of the domain are not expected to be significant. We perform time marching of (\ref{eq:trunca}) for the initial data
\begin{equation}
    u(x,0)=A\, \text{sech}^{2}\big(B (x+15) \big),~~~~ u_{t}(x,0)=2A\, B\, c\, \text{sech}^{2}\big(B (x+15) \big)\tanh \big(B (x+15) \big)
\end{equation}
with $c=1.5$ up to time $ t= 20$ over the computational domain $[-30, 30]$. With this empirical choice  of the computational domain, the truncation will yield negligible influence on the numerical results. This is due to the exponential decay of the solution; then the proof of Theorem \ref{theo:theo5.2} shows that $\epsilon={\cal O}\left(e^{-CNh}\right)$. The right and left boundaries are chosen so that the distance from the right boundary to the "crest" of the solitary wave at $t=20$ is equal to the distance from the left boundary  to the "crest" of the solitary wave at $t=0$.  We perform a discretization of the computational domain  with equal spatial step size $h$  in which $2N$ is the number of subintervals.  Figure \ref{fig:Fig1} compares  the exact and approximate solutions at $t=20$ when $h=0.125$ ($N=240$).  As can be seen from the figure, the exact and approximate solutions are almost indistinguishable. This experiment validates the ability of the semi-discrete scheme proposed to adequately capture the propagation of a single solitary wave for (\ref{eq:cont}).
\begin{figure}[h!]
    \centering
 %   \subfloat[$\epsilon_{\max}=5$]{\label{fig:a}\includegraphics{lexample_fig1}}
 %   \subfloat[$\epsilon_{\max}=0.5$]{\label{fig:b}\includegraphics{lexample_fig2}}
    \includegraphics[width=0.80\linewidth,scale=1.50,keepaspectratio]{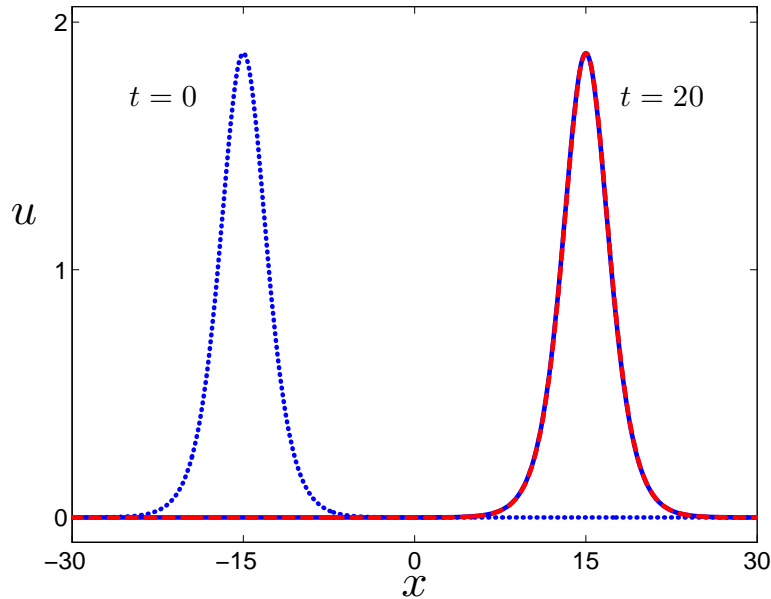}
    \caption{Propagation of a right-moving solitary wave of speed $c=1.5$ for the nonlocal nonlinear wave equation (\ref{eq:cont}) with $\beta(x)={\frac{1}{2}}e^{-|x|}$ and $f(u)=u+u^{2}$. The initial profile, the exact and the numerical solutions at   $t_{\text{max}}=20$   are shown with the dotted line, the solid line and the dashed line, respectively. The numerical solution is almost indistinguishable from the exact solution. For the numerical simulation the computational domain $[-30, \, 30]$ and the mesh size   $h=0.125$ are used.}
    \label{fig:Fig1}
\end{figure}

We now perform numerical simulations with different values of $h$ and $N$ to illustrate the convergence rate in space. In the first set of our numerical experiments we fix the computational domain to $[-30, 30]$ and change the mesh size $h$. A summary of the results of these experiments is given in  Table \ref{tab:tab1}, where the $l^{\infty}$-error $E_{h}^{N}$ at time $t=20$ and the convergence rate $\rho$ are given. The  $l^{\infty}$-error $E_{h}^{N}$ at time $t$ and the  convergence rate $\rho$ are calculated as
\begin{equation}
    E_{h}^{N}(t)=\left\Vert \mathbf{u}(t) - \mathbf{u}_{h}^{N}(t) \right\Vert_{l^{\infty}}
                =\max_{-N\leq i \leq N} \left \vert u(x_{i},t)- (\mathbf{u}_{h}^{N})_{i}\right \vert
 \end{equation}
 and
 \begin{equation}
    \rho={1\over {\ln 2}}\ln\left({{E_{h}^{N}(t)}\over {E_{2h}^{2N}(t)}}\right),
 \end{equation}
 respectively. We observe that the errors decrease rapidly as the mesh size $h$ decreases.  We also observe that numerical simulations  validate the second-order accuracy in space, which is  the  convergence rate predicted by Theorem \ref{theo:theo5.2}.
\begin{table}[tbhp]
    \caption{Variation of the $l^{\infty}$-error ($E_{h}^{N}$) computed at time $t=20$ with the mesh size ($h$) and  the corresponding convergence rates ($\rho$). Results are provided for the  solitary wave problem of the nonlocal nonlinear wave equation (\ref{eq:cont}) with the wave speed  $c=1.5$, the kernel $\beta(x)={\frac{1}{2}}e^{-|x|}$ and the quadratic nonlinearity  $f(u)=u+u^{2}$ when   the computational domain is $[-30, \, 30]$.}
    \label{tab:tab1}
   \centering
   \begin{tabular}{|c|c|c|} \hline
          $h$          & $E_{h}^{N}$ & Order \\ \hline
          2.00000   & 1.37663752E+00 & - \\
          1.00000   & 5.40121525E-01 & 1.3497 \\
          0.50000   & 1.47892030E-01 & 1.8687 \\
          0.25000   & 3.75864211E-02 & 1.9762 \\
          0.12500   & 9.43402186E-03 & 1.9942 \\
          0.06250   & 2.36067921E-03 & 1.9986 \\
          0.03125   & 5.90372954E-04 & 1.9995 \\ \hline
   \end{tabular}
\end{table}
In the second set of the numerical experiments, we  carry out the same experiments as above for different values of $N$ but with a fixed mesh size $h$. Thus, the size of the computational domain is not the same for each experiment and  its length increases as $N$ increases. A summary of the results of these experiments is given in  Table \ref{tab:tab2}, where, for each experiment, the computational domain and the $l^{\infty}$-errors ($E_{h}^{N}$) at times $t=5,10,15, 20$ are presented.  At each time in the table, we observe   a similar monotonically decreasing behavior of $l^{\infty}$-error with increasing $N$. Figure \ref{fig:Fig2} displays, on a semi-logarithmic scale, both the variation of $l^{\infty}$-error at time $t=20$ and  the variation of the maximum of the $l^{\infty}$-errors at times $t=5,10,15, 20$ with $N$ for the fixed mesh size. In the figure the two curves are almost indistinguishable. We again observe that the logarithms of the errors  decrease linearly as  $N$ increases up to a certain value of $N$ ($ \approx 220$). This is in complete agreement with our previous observation that $E_{h}^{N}={\cal O}\left(h^{2}+e^{-CNh}\right)$. Additionally we observe that above a certain value of $N$,  an increase in $N$ does not affect substantially the accuracy of the results,  which shows that cut-off errors are negligible in comparison with discretization errors.
\begin{table}[tbhp]
    \caption{Variation of the $l^{\infty}$-errors ($E_{h}^{N}$) computed at times $t=5, 10, 15, 20$ with the computational domain. The last column presents the maximum of the $l^{\infty}$-errors at times $t=5,10,15, 20$. Results are provided for the  solitary wave problem of the nonlocal nonlinear wave equation (\ref{eq:cont}) with the wave speed  $c=1.5$, the kernel $\beta(x)={\frac{1}{2}}e^{-|x|}$ and the quadratic nonlinearity  $f(u)=u+u^{2}$  when   the mesh size is $h=0.1$.}
    \label{tab:tab2}
   \centering
   \begin{tabular}{|c|c|c|c|c|c|c|c|} \hline
   $N$ &  Domain          &  $t=5$           &  $t=10$           &  $t=15$          &  $t=20$         &  Maximum \\ \hline
   160 &  $[-16, \, 16]$  &  5.696E-01  &  2.568E-01   &  2.771E-01  &  5.834E-01 &  5.834E-01  \\
   180 &  $[-18, \, 18]$  &  1.043E-01  &  7.893E-02   &  6.767E-02  &  1.127E-01 &  1.127E-01   \\
   200 &  $[-20, \, 20]$  &  2.369E-02  &  1.852E-02   &  1.606E-02  &  2.345E-02 &  2.369E-02       \\
   220 &  $[-22, \, 22]$  &  4.937E-03  &  4.203E-03   &  4.583E-03  &  6.038E-03 &  6.038E-03       \\
   240 &  $[-24, \, 24]$  &  1.702E-03  &  3.136E-03   &  4.586E-03  &  6.040E-03 &  6.040E-03       \\
   260 &  $[-26, \, 26]$  &  1.701E-03  &  3.136E-03   &  4.586E-03  &  6.040E-03 &  6.040E-03          \\
   280 &  $[-28, \, 28]$  &  1.701E-03  &  3.136E-03   &  4.586E-03  &  6.040E-03 &  6.040E-03        \\ \hline
   \end{tabular}
\end{table}
\begin{figure}[h!]
    \centering
    \includegraphics[width=0.80\linewidth,scale=1.50,keepaspectratio]{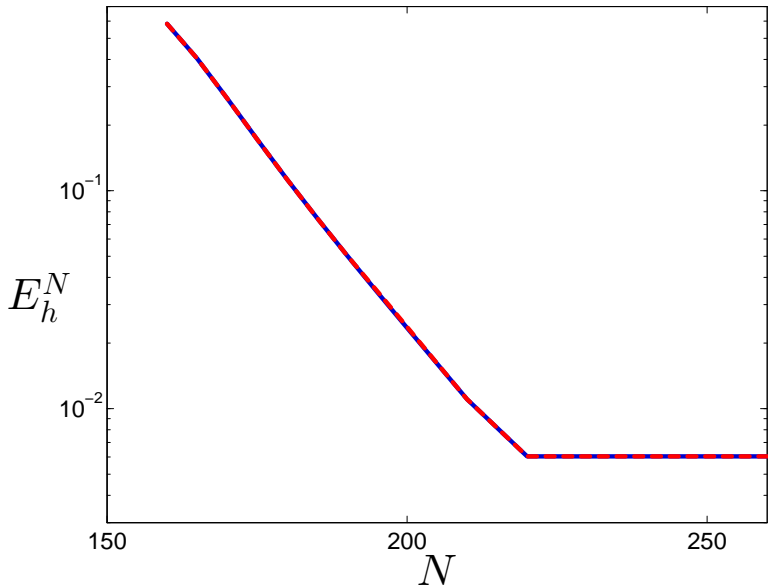}
    \caption{Variation of  the $l^{\infty}$-error ($E_{h}^{N}$) with $N$ for the  solitary wave problem of the nonlocal nonlinear wave equation (\ref{eq:cont}) with the wave speed  $c=1.5$, the kernel $\beta(x)={\frac{1}{2}}e^{-|x|}$  and the quadratic nonlinearity  $f(u)=u+u^{2}$. The $l^{\infty}$-error at time $t=20$ and the maximum of the $l^{\infty}$-errors at times $t=5, 10, 15, 20$ are shown with  the solid line and the dashed line, respectively. The two curves are almost indistinguishable. For all the numerical simulations the  mesh size  is fixed at $h=0.1$. }
    \label{fig:Fig2}
\end{figure}

\subsection{Blow-up}

In \cite{Duruk2010}, it was rigorously proved that  the class (\ref{eq:cont}) and its member (\ref{eq:imbq}) have finite time blow-up solutions under appropriate initial conditions (see Theorem 5.2 in  \cite{Duruk2010}). Moreover, as we stated in Remark \ref{rem:blow},   one can not have higher-order singularities if the amplitude stays finite. That is, if blow-up occurs it should be observed in $\left\Vert u\left( t\right) \right\Vert _{L^{\infty} }$. For this reason, in the following experiments  checking the magnitude  $\left\Vert u\left( t\right) \right\Vert _{l^{\infty} }$ suffices to show blow-up.

It is clear that, because of the infinitely large spatial and temporal gradients that exist near the blow-up time, the numerical simulation of  (\ref{eq:cont}) near the blow-up time and the numerical identification of the blow-up time are  challenging tasks. To see the semi-discrete method at work for finite time blow-up solutions we now consider  (\ref{eq:cont})   with quadratic nonlinearity and the initial conditions
\begin{equation}
    u(x,0)=4\left({2\over 3}x^{2}-1\right)e^{-x^{2}/3}, ~~~~~ u_{t}(x,0)=\left(x^{2}-1\right)e^{-x^{2}/2}. \label{eq:blowin}
\end{equation}
It is  known  that the solution of the initial-value problem (\ref{eq:cont}) and (\ref{eq:blowin})  blows up in a finite time for the exponential kernel (that is, for $\beta_{1}$ given below) \cite{Godefroy1998} and the blow-up time is about $1.8$ \cite{Borluk2017}. Our first goal  is to validate the semi-discrete method for the blow-up solution corresponding to the exponential kernel and the above initial data. Our next goal is to show  that, for several other types of kernel functions, the  solution corresponding to the same initial data blows up in a finite time.

We conduct the numerical experiments using  the following four different kernel functions:
\begin{align}
   & \mbox{(a)} \hspace*{20pt} \beta_{1}(x)=\frac{1}{2}e^{-|x|}, \label{eq:beta1} \\
   & \mbox{(b)} \hspace*{20pt}\beta_{2}(x)=\frac{1}{\pi}{\frac{1}{1+x^2}}, \label{eq:beta2} \\
   & \mbox{(c)} \hspace*{20pt}\beta_{3}(x)=\frac{1}{e^{x}+e^{-x}+2}, \label{eq:beta3} \\
   & \mbox{(d)} \hspace*{20pt}\beta_{4}(x)=\left\{ \begin{array}{ll} 1-|x|, & \mbox{if } |x|\leq 1 \\ 0, & \mbox{if } |x|> 1. \end{array} \right. \label{eq:beta4}
\end{align}%
We note that for $\beta_{4}$  (\ref{eq:cont}) reduces to the difference-differential equation arising in lattice dynamics;
\begin{equation}
    u_{tt}=f\big(u(x+1,t)\big)-2f\big(u(x,t)\big)+f\big(u(x-1,t)\big). \label{eq:lattice}
\end{equation}
All of the above kernels are symmetric and nonnegative functions of $x$.   While $\beta_{4}$ is a function with compact support on
$ \lbrack -1, 1\rbrack$, the kernels $\beta_{2}$ and $\beta_{1}, \beta_{3}$ are $C^{\infty}$ smooth functions that  decay algebraically and exponentially, respectively. We also note that the second derivatives of $\beta_{1}$ and $\beta_{4}$ involve  the delta functions at $x=0$ and at $x=-1,0,1$, respectively. So,  the order of regularization will not be the same for all these kernels and we  may list the above kernel functions in increasing  order of regularization as $\beta_{4}$, $\beta_{1}$, $\beta_{2}$ and $\beta_{3}$. We expect that  possible blow-up times increase with increasing effect of regularization.

For each of the kernel functions in (\ref{eq:beta1})-(\ref{eq:beta4}) we numerically integrate (\ref{eq:cont}) under the initial conditions  (\ref{eq:blowin}) using the Matlab solver \verb"ode45" to solve the semi-discrete scheme (\ref{eq:trunca}).  As in the previous example,  we verify that the  relative and absolute tolerances  for the Matlab solver \verb"ode45"  are sufficiently small so that the temporal discretization has no significant effect on the numerical results.  In Figure \ref{fig:Fig3} we present the approximate numerical results obtained for $\left \Vert u(t) \right \Vert_{L^{\infty}}$ using the mesh size $h=0.1$ and the computational domain $[-10, 10]$ (that is,  we present $\left \Vert \mathbf{u}_{h}^{N}(t) \right \Vert_{l^{\infty}}$ for $N=100$). Figure \ref{fig:Fig3} displays, on a semi-logarithmic scale, the variation of $\left \Vert u(t) \right \Vert_{L^{\infty}}$  with time for the  kernel functions introduced above.   For each of the kernel functions, we observe that the magnitude of $\left \Vert u(t) \right \Vert_{L^{\infty}}$ becomes arbitrarily large as time approaches the blow-up time. From these numerical experiments, we conclude that all of the   kernel functions  produce a blowing-up solution (of course, the solutions blow-up at different times). We estimate the blow-up time $t^{*}$  to be approximately $t^{*}=1.804484$, $t^{*}=2.689993$, $t^{*}=4.396459$ and $t^{*}=1.135569$ for the kernel functions $\beta_{i}$ ($i=1,2,3,4$), respectively. We observe that the order of these blow-up times are in complete agreement with the ordering based on regularization effects of  the kernels.
\begin{figure}[h!]
    \centering
    \subfloat[$\beta_{1}(x)={\frac{1}{2}}e^{-|x|}$]{\label{fig:a}\includegraphics[width=0.47\linewidth,scale=1.50,keepaspectratio]{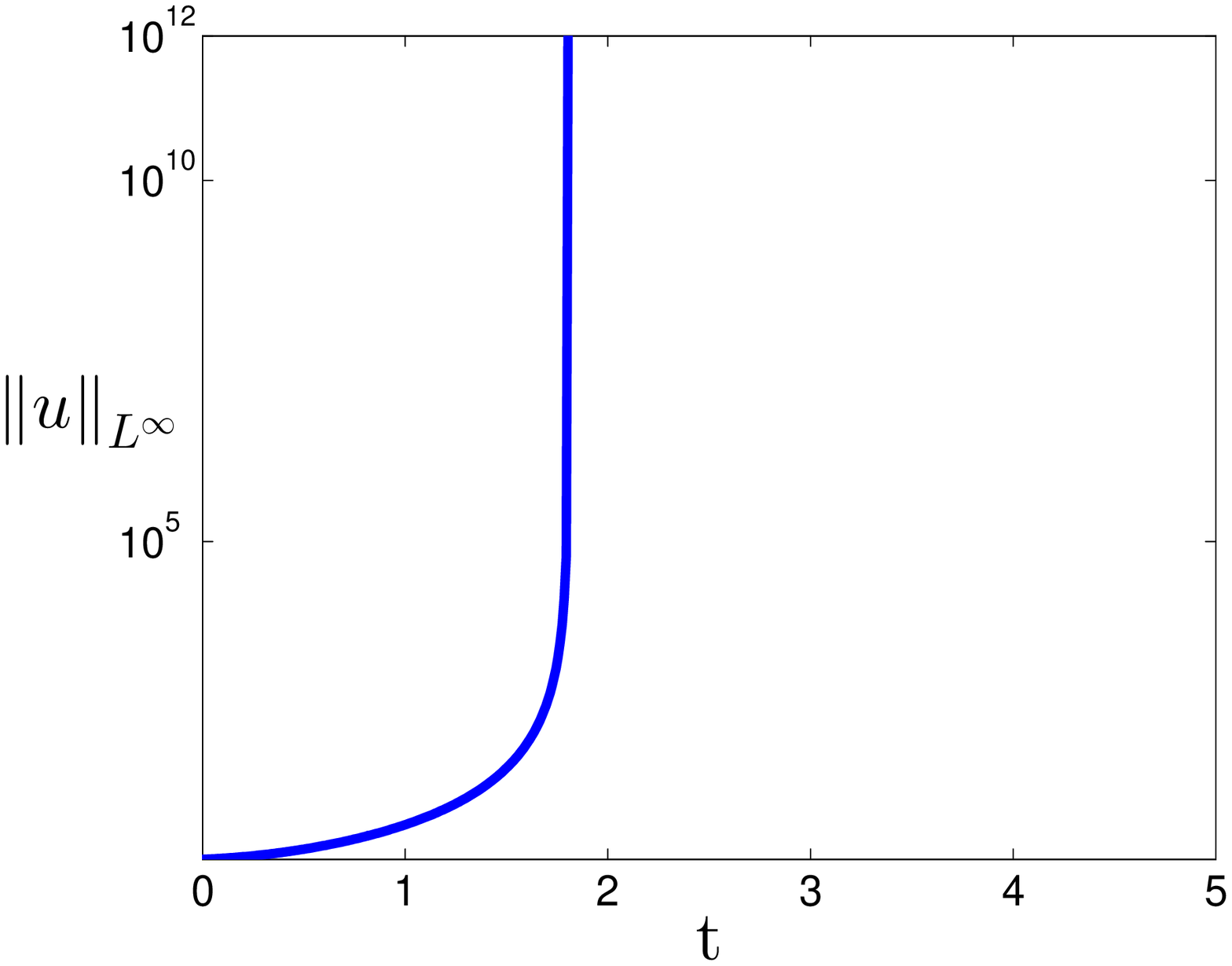}}\hspace*{10pt}
    \subfloat[$\beta_{2}(x)=\frac{1}{\pi}{\frac{1}{1+x^2}}$]{\label{fig:b}\includegraphics[width=0.47\linewidth,scale=1.50,keepaspectratio]{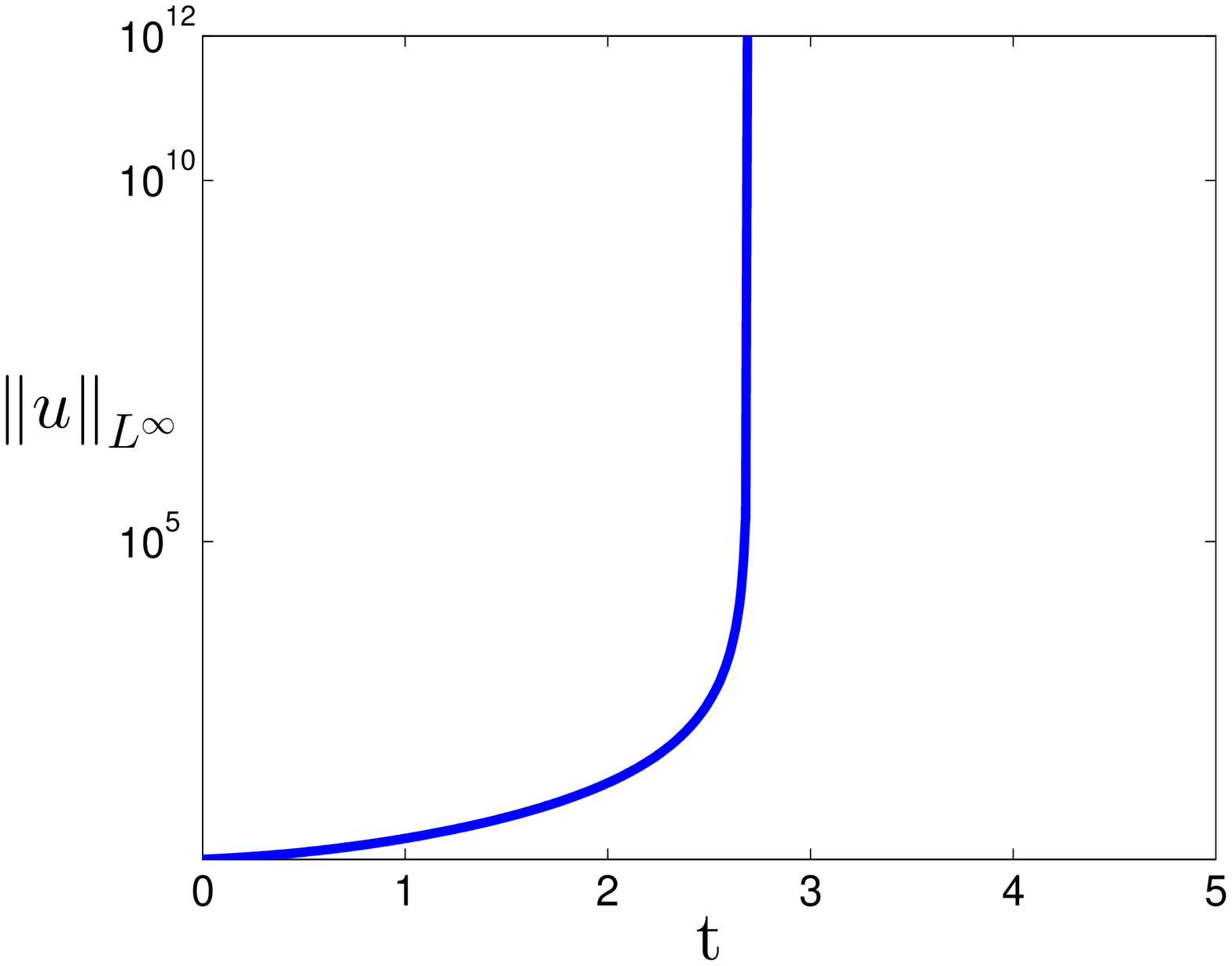}}\\
    \subfloat[$\beta_{3}(x)=\frac{1}{e^{x}+e^{-x}+2}$]{\label{fig:c}\includegraphics[width=0.47\linewidth,scale=1.50,keepaspectratio]{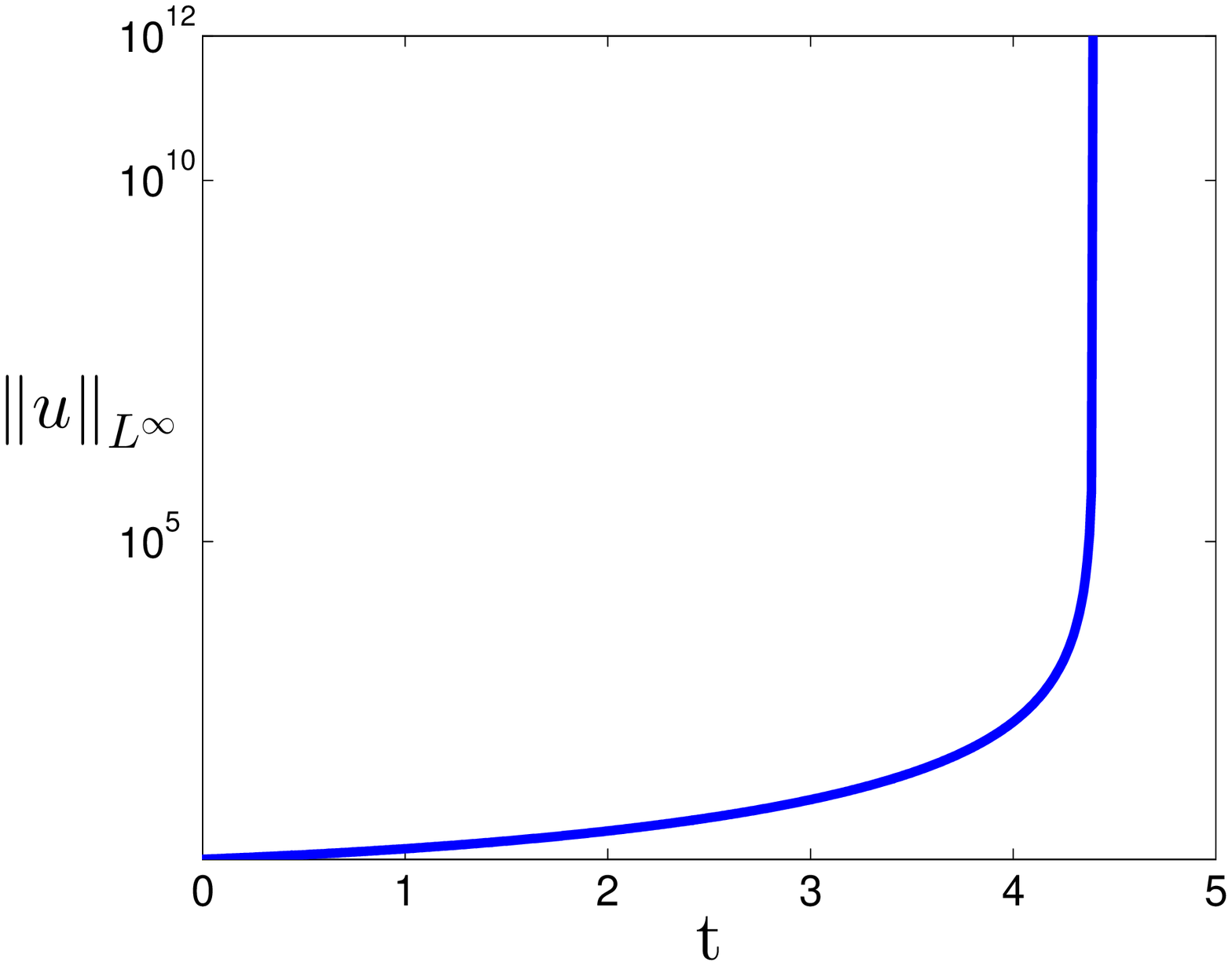}}\hspace*{10pt}
    \subfloat[$\beta_{4}(x) =  1-|x|$ if $|x|\leq 1$ and $0$ if $|x|> 1$]{\label{fig:d}\includegraphics[width=0.47\linewidth,scale=1.50,keepaspectratio]{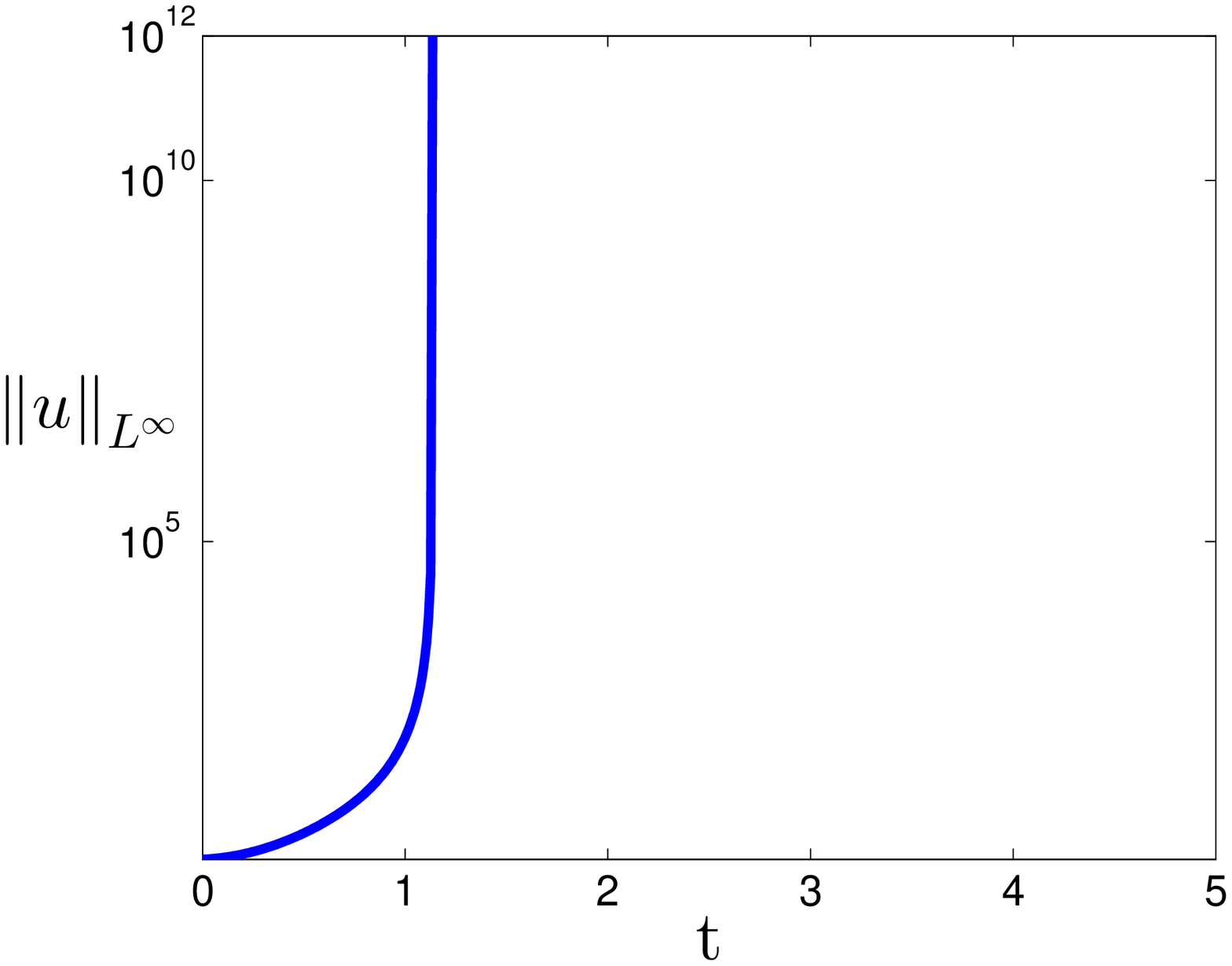}}
    \caption{Numerical approximations to the blow-up solutions of the nonlocal nonlinear wave equation (\ref{eq:cont}) for the kernel functions $\beta_{i}(x)$ ($i=1,2,3,4$) and the quadratic nonlinearity  $f(u)=u+u^{2}$. Time evolution of $\left \Vert u(t) \right \Vert_{L^{\infty}}$  is estimated by the semi-discrete scheme using the computational domain $[-10, 10]$ and the mesh size $h=0.1$. In the above plots,  the critical times when singularities develop are approximately as follows: (a) $t^{*}=1.804484$, (b) $t^{*}=2.689993$, (c) $t^{*}=4.396459$ and (d)  $t^{*}=1.135569$.}
    \label{fig:Fig3}
\end{figure}

To investigate numerically the convergence of the semi-discrete scheme in terms of the mesh size $h$ we again consider (\ref{eq:cont}) with the exponential kernel  given in (\ref{eq:beta1}). We then  solve  the equation under the initial conditions given by (\ref{eq:blowin}) for various values of $h$. Figure \ref{fig:Fig4} displays, on a semi-logarithmic scale, the variation of $\left \Vert u(t) \right \Vert_{L^{\infty}}$  with time when $h=10/N$ with $N=2,5,10,20,40,80,100$ and the computational domain is $[-10, 10]$.  In Figure \ref{fig:Fig4}, the eight curves  from right to left correspond to $N=2,5,10,20,40,80$ and $100$, respectively and the curves corresponding to $N=20,40,60,80$ and $100$ are indistinguishable. This shows that the blow-up times obtained for various $N$ converge  rapidly to the blow-up time obtained for the finest grid ($N=100$) as $N$ increases.
\begin{figure}[h!]
    \centering
 %   \subfloat[$\epsilon_{\max}=5$]{\label{fig:a}\includegraphics{lexample_fig1}}
 %   \subfloat[$\epsilon_{\max}=0.5$]{\label{fig:b}\includegraphics{lexample_fig2}}
    \includegraphics[width=0.80\linewidth,scale=1.50,keepaspectratio]{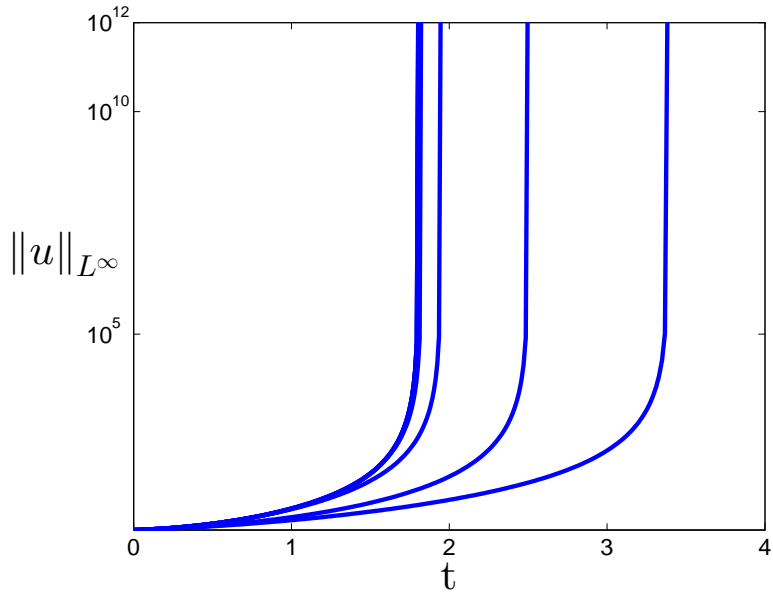}
    \caption{Convergence of the numerical approximations to the blow-up solution of the nonlocal nonlinear wave equation (\ref{eq:cont}) for  the kernel function $\beta(x)={\frac{1}{2}}e^{-|x|}$ and the quadratic nonlinearity  $f(u)=u+u^{2}$ as the mesh size $h$ decreases. Time evolution of $\left \Vert u(t) \right \Vert_{L^{\infty}}$  is estimated by the semi-discrete scheme using the computational domain $[-10, 10]$ and various mesh sizes. Each curve corresponds to  one of the different values of $h(=10/N)$, decreasing from right to left. In terms of $N$, the eight curves from right to left correspond to  $N=2,5,10,20,40,60,80$ and $100$, respectively. The curves corresponding to $N=20,40,60,80$ and $100$ are indistinguishable. The blow-up time is estimated approximately $t^{*}=1.804484$. }
    \label{fig:Fig4}
\end{figure}

\appendix
\setcounter{equation}{0}
\section{Proofs of Lemmas \ref{lem:lem2.1}, \ref{lem:lem2.6} and \ref{lem:lem2.7}}\label{sec:appendixA}
In this appendix, for completeness, we present  the technical proofs of Lemmas \ref{lem:lem2.1}, \ref{lem:lem2.6} and \ref{lem:lem2.7} about discretization errors for integrals on $\mathbb{R}$ and derivatives.
\subsection{Proof of Lemma \ref{lem:lem2.1}}
\begin{proof}
    We have
    \begin{displaymath}
         u(x)-u_{i} = \int_{x_{i}}^{x}  u^{\prime}(s)ds.
    \end{displaymath}%
    By Holder's inequality
    \begin{displaymath}
        \left\vert u(x)-P_{0}(\mathbf{u})(x)\right\vert ^{p}=\left\vert \int_{x_{i}}^{x}u^{\prime }(s)ds\right\vert^{p}
            \leq \left\vert x-x_{i}\right\vert^{\frac{p}{q}}\int_{x_{i}}^{x}\left\vert u^{\prime }(s)\right\vert ^{p}ds,
    \end{displaymath}%
    where $q$ is the dual exponent to $p$. Then
    \begin{align*}
        \Vert u-P_{0}(\mathbf{u})\Vert _{L^{p}}^{p}
            & \leq h^{\frac{p}{q}}\sum_{i}\int_{x_{i}}^{x_{i+1}}\int_{x_{i}}^{x}|u^{\prime }(s)|^{p}dsdx  \\
            & \leq h^{\frac{p}{q}+1}\sum_{i}\int_{x_{i}}^{x_{i+1}}|u^{\prime}(s)|^{p}ds  \\
            & =h^{p}\Vert u^{\prime }\Vert _{L^{p}}^{p}.
    \end{align*}%
     Furthermore, since $\Vert P_{0}(\mathbf{u})\Vert_{L^{p}}=\Vert \mathbf{u}\Vert_{l_{h}^{p}}$, we get $\mathbf{u}\in l_{h}^{p}$. To prove  (\ref{eq:est2}) we start with the identity
    \begin{displaymath}
        u(x)-P_{1}(\mathbf{u})(x)=\frac{1}{h}\int_{x_{i}}^{x}\int_{x_{i}}^{x_{i+1}}\int_{r}^{s}u^{\prime \prime }(\theta )d\theta drds,
    \end{displaymath}
    for $x_{i}\leq x<x_{i+1}$. For $\ \ x_{i}\leq x, r, s<x_{i+1}$ we have
    \begin{displaymath}
       \left\vert u(x)-P_{1}(\mathbf{u})(x)\right\vert
        \leq \frac{1}{h}\int_{x_{i}}^{x_{i+1}}\int_{x_{i}}^{x_{i+1}}\int_{x_{i}}^{x_{i+1}}\left\vert u^{\prime \prime }(\theta)\right\vert d\theta drds
        =h\int_{x_{i}}^{x_{i+1}}\left\vert u^{\prime \prime }(\theta )\right\vert d\theta .
    \end{displaymath}%
    Consequently
    \begin{displaymath}
       \left\vert u(x)-P_{1}(\mathbf{u})(x)\right\vert^{p}
        \leq h^{p(q+1)/q}\int_{x_{i}}^{x_{i+1}}\left\vert u^{\prime \prime }(\theta )\right\vert^{p} d\theta
    \end{displaymath}%
    and
    \begin{displaymath}
        \Vert u-P_{1}(\mathbf{u})\Vert _{L^{p}}\leq h^{2}\Vert u^{\prime \prime}\Vert _{L^{p}}.
    \end{displaymath}
\end{proof}
\subsection{Proof of Lemma \ref{lem:lem2.6}}
\begin{proof}
    We prove the lemma for $D^{+},$ the proof for $D^{-}$ being similar. Since
    \begin{displaymath}
        (D^{+}\mathbf{u})_{i}-(\mathbf{u}^{\prime })_{i}=\frac{1}{h}\int_{x_{i}}^{x_{i+1}}\left(x_{i+1}-s\right)u^{\prime \prime }(s)ds,
    \end{displaymath}%
    then
    \begin{displaymath}
        \left \vert (D^{+}\mathbf{u})_{i}-(\mathbf{u}^{\prime})_{i}\right \vert
        \leq \frac{1}{h}\int_{x_{i}}^{x_{i+1}}\left(x_{i+1}-s\right)\left \vert u^{\prime \prime}(s)\right \vert ds
        \leq \frac{h}{2}\Vert u^{\prime \prime }\Vert _{L^{\infty}}.
    \end{displaymath}%
\end{proof}
\subsection{Proof of Lemma \ref{lem:lem2.7}}
\begin{proof}
   From the Taylor expansion we have
    \begin{displaymath}
        \left(D^{+}D^{-}\mathbf{u}\right)_{i}-(\mathbf{u}^{\prime \prime })_{i}=\frac{1}{h^{2}}\int_{x_{i-1}}^{x_{i+1}}Q(s)u^{(4)}(s)ds,
    \end{displaymath}%
    where
    \begin{displaymath}
    Q(s)=\left\{\begin{array}{ll}
         \frac{1}{3!}(s-x_{i-1})^{3}, & ~~x_{i-1}\leq s<x_{i},  \\ ~ \\
        \frac{1}{3!}(x_{i+1}-s)^{3}, & ~~x_{i}\leq s<x_{i+1}.
        \end{array} \right.
    \end{displaymath}%
    Then
    \begin{displaymath}
        \left \vert \left(D^{+}D^{-}\mathbf{u}\right)_{i}-(\mathbf{u}^{\prime \prime })_{i}\right \vert
        \leq \frac{1}{h^{2}}  \Vert u^{(4)}\Vert _{L^{\infty}} \int_{x_{i-1}}^{x_{i+1}}Q(s)ds
        \leq \frac{h^{2}}{12}\Vert u^{(4)}\Vert _{L^{\infty}}. \nonumber
    \end{displaymath}%
\end{proof}

\setcounter{equation}{0}
\section{Decay Estimates}\label{sec:appendixB}

As it was mentioned in Remark \ref{rem:rem5.4}, for particular kernels it is possible to obtain decay estimates for the solution.  The next lemma provides such a decay estimate.
\begin{lemma}\label{lem:lemA.1}
    Let $\omega\left(x\right) $ be a positive function such that $\left(\left\vert \beta ^{\prime \prime }\right\vert \ast \omega\right)(x) \leq C\omega(x) $ for all $x\in \mathbb{R}$. Suppose that $\varphi \omega^{-1},\psi \omega^{-1}\in L^{\infty }\left(\mathbb{R}\right) $. The solution $u\in C^{2}\left([0,T], H^{s}(\mathbb{R})\right)$ of (\ref{eq:cont1})-(\ref{eq:initial}) then satisfies the estimate
    \begin{displaymath}
        \left\vert u(x,t)\right\vert \leq C\omega\left( x\right)
    \end{displaymath}
    for all  $x\in \mathbb{R}$,  $t\in \left[ 0,T\right] $.
\end{lemma}
\begin{proof}
    It will be convenient to express the nonlinear term as $f(u)=k(u)u$. Let $M$ and $K_{M}$ be $M=\max \left\{ \left\vert u(x,t)\right\vert :t\in \left[    0,T\right], \text{ } x\in \mathbb{R}\right\} $ and $K_{M}=\max \left\{\left\vert k(u)\right\vert :\text{ }\left\vert u\right\vert \leq M\right\} $, respectively. Since $u$ satisfies
    \begin{equation}
        u\left( x,t\right) =\varphi \left( x\right) +t\psi \left( x\right)
            +\int_{0}^{t}\left( t-\tau \right) \Big(\beta^{\prime \prime }\ast \big(k(u) u\big)\Big)(x, \tau) d\tau ,  \label{eq:app1}
    \end{equation}%
    letting $v\left( x,t\right) =\omega^{-1}\left( x\right) u\left( x,t\right)$, we get
    \begin{displaymath}
        v\left( x,t\right) =\varphi \left( x\right) \omega^{-1}\left( x\right)
            +t\psi \left( x\right) \omega^{-1}\left( x\right)
            +\omega^{-1}\left( x\right)\int_{0}^{t}\left( t-\tau \right) \Big(\beta^{\prime \prime}\ast \big( k(u) \omega v\big)\Big) (x,\tau) d\tau.
    \end{displaymath}
    This gives the estimate
    \begin{equation}
        \left\Vert v\left( t\right) \right\Vert _{L^{\infty }}
        \leq \left\Vert \varphi \omega^{-1}\right\Vert _{L^{\infty }}
            +t\left\Vert \psi \omega^{-1}\right\Vert_{L^{\infty }}
            +\omega^{-1}\left( x\right) \int_{0}^{t}\left( t-\tau \right)
                \left\Vert \beta^{\prime \prime }\ast \big( k(u) \omega v\big)(\tau)\right\Vert _{L^{\infty }} d\tau . \label{eq:app2}
    \end{equation}
    But
    \begin{displaymath}
        \Big(\beta^{\prime \prime }\ast \big(k(u) \omega v\big)\Big) \left( x,\tau\right)
        =\int \beta ^{\prime \prime }\left( y\right) k\big( u\left(x-y,\tau \right) \big) \omega \left( x-y\right) v\left( x-y,\tau \right) dy,
    \end{displaymath}
    so that
    \begin{align*}
        \left\vert \Big(\beta^{\prime \prime }\ast \big(k(u) \omega v\big)\Big) \left(x,\tau \right) \right\vert
        &\leq K_{M}\left\Vert v(\tau) \right\Vert_{L^{\infty }}\int \left\vert \beta^{\prime\prime }(y) \right\vert  \omega( x-y) dy \\
        &=K_{M}\left\Vert v\left( \tau \right) \right\Vert _{L^{\infty }}\left(\left\vert\beta ^{\prime \prime }\right\vert \ast \omega\right)\left( x\right)  \\
        &\leq CK_{M}\left\Vert v(\tau) \right\Vert _{L^{\infty}}\omega(x) .
    \end{align*}%
    We note that when $\beta ^{\prime \prime }=\mu $ is a finite measure the convolution integral above should be interpreted as
    \begin{displaymath}
        \Big(\beta^{\prime \prime }\ast \big(k(u) \omega v\big)\Big) (x, \tau) =\int k\big( u(x-y,\tau) \big) \omega( x-y)v(x-y,\tau) d\mu (y),
    \end{displaymath}
    but the estimate afterwards remains valid. Replacing this estimate in (\ref{eq:app2}), for $t\in \left[ 0,T\right] $ yields
    \begin{align*}
        \left\Vert v\left( t\right) \right\Vert _{L^{\infty }}
        &\leq \left\Vert \varphi \omega^{-1}\right\Vert_{L^{\infty }}
            +t\left\Vert \psi \omega^{-1}\right\Vert_{L^{\infty }}
            +CK_{M}\int_{0}^{t}\left( t-\tau \right) \left\Vert v(\tau) \right\Vert _{L^{\infty }}d\tau . \\
        &\leq \left\Vert \varphi \omega^{-1}\right\Vert _{L^{\infty }}
            +T\left\Vert \psi \omega^{-1}\right\Vert _{L^{\infty }}+CK_{M}T\int_{0}^{t}\left\Vert v(\tau) \right\Vert _{L^{\infty }}d\tau .
    \end{align*}%
    By Gronwall's lemma we have
    \begin{displaymath}
        \left\Vert v(t) \right\Vert _{L^{\infty }}\leq \big( \left\Vert\varphi \omega^{-1}\right\Vert _{L^{\infty }}
            +T\left\Vert \psi \omega^{-1}\right\Vert_{L^{\infty }}\big) e^{CK_{M}Tt}.
    \end{displaymath}
    which in turn implies
    \begin{displaymath}
        \left\vert u(x,t)\right\vert \leq C\omega (x)
    \end{displaymath}
    for all $t\in \left[ 0,T\right]$.
\end{proof}
We now give some examples for kernels $\beta $ and weights $\omega$ satisfying the condition in Lemma \ref{lem:lemA.1}.
\begin{example}[The improved Boussinesq equation]
      The IB equation corresponds to the kernel $\beta \left(x\right) =\frac{1}{2}e^{-\left\vert x\right\vert }$. Then $\beta ^{\prime\prime }=\beta -\delta$ with the Dirac measure $\delta$. Let $\omega\left( x\right) =e^{-r\left\vert x\right\vert }$ with $0<r<1$. Then
    \begin{displaymath}
        \left\vert \beta^{\prime \prime }\right\vert \ast \omega =\beta \ast \omega +\omega
    \end{displaymath}
    and
    \begin{align*}
        \left(\beta \ast \omega\right)(x)
        &=\frac{1}{2}\int e^{-\left\vert x-y\right\vert }e^{-r\left\vert y\right\vert }dy
        =\frac{1}{2}\int e^{-(1-r)\left\vert x-y\right\vert }e^{-r\left( \left\vert x-y\right\vert +\left\vert y\right\vert \right) }dy \\
        &\leq \frac{1}{2}e^{-r\left\vert x\right\vert }\int e^{-(1-r)\left\vert x-y\right\vert }dy
        =\frac{1}{1-r}e^{-r\left\vert x\right\vert }=\frac{1}{1-r}\omega(x)
    \end{align*}%
    so
    \begin{displaymath}
        \left(\left\vert \beta ^{\prime \prime }\right\vert \ast \omega\right)(x) \leq \left( \frac{1}{1-r}+1\right) \omega\left( x\right).
    \end{displaymath}
    Then, for initial data with $\varphi \left( x\right) e^{r\left\vert x\right\vert },$ $\psi (x) e^{r\left\vert x\right\vert }\in L^{\infty }\left( \mathbb{R}\right) $, solutions of the IB equation satisfy the decay estimate
    \begin{displaymath}
        \left\vert u(x,t)\right\vert \leq Ce^{-r\left\vert x\right\vert }, ~~~~~0<r<1
    \end{displaymath}
     for all $t\in \left[ 0,T\right]$.
\end{example}
\begin{example}[The triangular kernel]
    Let $\beta \left( x\right) =\left(1-\left\vert x\right\vert \right) \chi_{\left[ -1,1\right] }(x) $ where  $\chi_{\left[ -1,1\right] }$ is the characteristic function of the interval $\left[ -1,1\right] $. Then $\beta ^{\prime \prime }(x)=\delta(x+1)-2\delta(x)+\delta(x-1)$ with the shifted Dirac measures. Again let $\omega\left( x\right) =e^{-r\left\vert x\right\vert }$ with any $r>0$. Then
    \begin{align*}
        \left(\left\vert \beta ^{\prime \prime }\right\vert \ast \omega\right)(x)
        &=\omega\left( x+1\right) +2\omega\left( x\right) +\omega\left( x-1\right)  \\
        &=e^{-r\left\vert x+1\right\vert }+2e^{-r\left\vert x\right\vert }+e^{-r\left\vert x-1\right\vert }\leq Ce^{-r\left\vert x\right\vert }.
    \end{align*}%
    Then, for initial data satisfying $\varphi \left( x\right) e^{r\left\vert x\right\vert },$ $\psi \left( x\right) e^{r\left\vert x\right\vert }\in
    L^{\infty }\left( \mathbb{R}\right) $, solutions of (\ref{eq:cont1})-(\ref{eq:initial}) will satisfy
    \begin{displaymath}
        \left\vert u(x,t)\right\vert \leq Ce^{-r\left\vert x\right\vert }, ~~~~~r>0
    \end{displaymath}
     for all $t\in \left[ 0,T\right]$.     We note that in this case the non-local equation (\ref{eq:cont}) reduces to the difference-differential equation (\ref{eq:lattice}) arising in lattice dynamics.
\end{example}

%%-----------------------------
%%      your bibliography
%%-----------------------------
\bibliographystyle{amsplain}
\bibliography{arXiv-M2AN}

\providecommand{\bysame}{\leavevmode\hbox to3em{\hrulefill}\thinspace}
\providecommand{\MR}{\relax\ifhmode\unskip\space\fi MR }
\providecommand{\MRhref}[2]{%
  \href{http://www.ams.org/mathscinet-getitem?mr=#1}{#2}
}
\providecommand{\href}[2]{#2}
\begin{thebibliography}{10}

\bibitem{Amorim2013}
P.~Amorim and M.~Figueira, \emph{Convergence of a finite difference method for
  the {K}d{V} and modified {K}d{V} equations with {$L^2$} data}, Portugal Math.
  \textbf{70} (2013), 23--50.

\bibitem{Armstrong2010}
S.~Armstrong, S.~Brown, and J.~Han, \emph{Numerical analysis for a nonlocal
  phase field system}, Int. J. Numer. Anal. Model. Ser. B \textbf{1} (2010),
  1--19.

\bibitem{Bates2009}
P.~W. Bates, S.~Brown, and J.~L. Han, \emph{Numerical analysis for a nonlocal
  {A}llen-{C}ahn equation}, Int. J. Numer. Anal. Model. \textbf{6} (2009),
  33--49.

\bibitem{Benjamin1972}
T.~B. Benjamin, J.~L. Bona, and J.~J. Mahony, \emph{Model equations for long
  waves in nonlinear dispersive systems}, Philos. Trans. R. Soc. Lond. Ser. A:
  Math. Phys. Sci. \textbf{272} (1972), 47--78.

\bibitem{Bona1981}
J.~L. Bona, W.~G. Pritchard, and L.~R. Scott, \emph{An evaluation of a model
  equation water waves}, Philos. Trans. R. Soc. Lond. Ser. A: Math. Phys. Sci.
  \textbf{302} (1981), 457--510.

\bibitem{Borluk2015}
H.~Borluk and G.~M. Muslu, \emph{A {F}ourier pseudospectral method for a
  generalized improved {B}oussinesq equation}, Numer. Methods Partial
  Differential Equations \textbf{31} (2015), 995–--1008.

\bibitem{Borluk2017}
\bysame, \emph{Numerical solution for a general class of nonlocal nonlinear
  wave equations arising in elasticity}, ZAMM-Z. Angew. Math. Mech. \textbf{97}
  (2017), 1600--1610.

\bibitem{Bratsos2007}
A.~G. Bratsos, \emph{A second order numerical scheme for the improved
  {B}oussinesq equation}, Phys. Lett. A \textbf{370} (2007), 145--147.

\bibitem{Cons2002}
A.~Constantin and L.~Molinet, \emph{The initial value problem for a generalized
  {B}oussinesq equation}, Differential and Integral Equations \textbf{15}
  (2002), 1061--1072.

\bibitem{Du2013}
Q.~Du, L.~Tian, and X.~Zhao, \emph{A convergent adaptive finite element
  algorithm for nonlocal diffusion and peridynamic models}, SIAM J. Numer.
  Anal. \textbf{51} (2013), 1211--1234.

\bibitem{Duruk2010}
N.~Duruk, H.~A. Erbay, and A.~Erkip, \emph{Global existence and blow-up for a
  class of nonlocal nonlinear {C}auchy problems arising in elasticity},
  Nonlinearity \textbf{23} (2010), 107--118.

\bibitem{Emmrich2007b}
E.~Emmrich and O.~Weckner, \emph{Analysis and numerical approximation of an
  integro-differential equation modelling non-local effects in linear
  elasticity}, Math. Mech. Solids. \textbf{12} (2007), 363--384.

\bibitem{Emmrich2007a}
\bysame, \emph{The peridynamic equation and its spatial discretisation}, Math.
  Model. Anal. \textbf{12} (2007), 17--27.

\bibitem{Godefroy1998}
A.~Godefroy, \emph{Blow-up solutions of a generalized {B}oussinesq equation},
  IMA J. Numer. Anal. \textbf{60} (1998), 122–--138.

\bibitem{Guan2015}
Q.~Guan and M.~Gunzburger, \emph{Stability and accuracy of time-stepping
  schemes and dispersion relations for a nonlocal wave equation}, Numer.
  Methods Partial Differential Equations \textbf{31} (2015), 500--516.

\bibitem{Meyer1997}
Y.~Meyer and R.~Coifman, \emph{Wavelets: Calder{\'o}n-{Z}ygmund and
  {M}ultilinear {O}perators}, Cambridge Studies in Advanced Mathematics,
  Cambridge University Press, 1997.

\bibitem{Oruc2017}
G.~Oruc, H.~Borluk, and G.~M. Muslu, \emph{Higher order dispersive effects in
  regularized {B}oussinesq equation}, Wave Motion \textbf{68} (2017),
  272–--282.

\bibitem{Rossi2011}
M.~Perez-LLanos and J.~D. Rossi, \emph{Numerical approximations for a nonlocal
  evolution equation}, SIAM J. Numer. Anal. \textbf{49} (2011), 2103–--2123.

\bibitem{Wang2014}
Q.~X. Wang, Z.~Y. Zhang, X.~H. Zhang, and Q.~Y. Zhu, \emph{Energy-preserving
  finite volume element method for the improved {B}oussinesq equation}, J.
  Comput. Phys. \textbf{270} (2014), 58--69.

\bibitem{Zhang2012}
Z.~Y. Zhang and F.~Q. Lu, \emph{Quadratic finite volume element method for the
  improved {B}oussinesq equation}, J. Math. Phys. \textbf{53} (2012),
  no.~013505.

\bibitem{Du2010}
K.~Zhou and Q.~Du, \emph{Mathematical and numerical analysis of linear
  peridynamic models with nonlocal boundary conditions}, SIAM J. Numer. Anal.
  \textbf{48} (2010), 1759–--1780.

\end{thebibliography}
\end{document}